\documentclass[a4paper,11pt]{amsart}

\usepackage[colorinlistoftodos,bordercolor=orange,backgroundcolor=orange!20,linecolor=orange,textsize=small]{todonotes}

\usepackage{amsmath}  
\usepackage{amssymb}
\usepackage{amsthm}
\usepackage{amsfonts}
     
\usepackage{bbm}

\usepackage{hyperref}
\usepackage{mathrsfs}
\usepackage{mathtools} 
\usepackage[inline]{enumitem}
\usepackage{fixmath}
\usepackage{fullpage}
\usepackage[capitalize]{cleveref}

\crefformat{equation}{\textup{#2(#1)#3}}
\crefrangeformat{equation}{\textup{#3(#1)#4--#5(#2)#6}}
\crefmultiformat{equation}{\textup{#2(#1)#3}}{ and \textup{#2(#1)#3}}
{, \textup{#2(#1)#3}}{, and \textup{#2(#1)#3}}
\crefrangemultiformat{equation}{\textup{#3(#1)#4--#5(#2)#6}}%
{ and \textup{#3(#1)#4--#5(#2)#6}}{, \textup{#3(#1)#4--#5(#2)#6}}{, and \textup{#3(#1)#4--#5(#2)#6}}

\Crefformat{equation}{#2Equation~\textup{(#1)}#3}
\Crefrangeformat{equation}{Equations~\textup{#3(#1)#4--#5(#2)#6}}
\Crefmultiformat{equation}{Equations~\textup{#2(#1)#3}}{ and \textup{#2(#1)#3}}
{, \textup{#2(#1)#3}}{, and \textup{#2(#1)#3}}
\Crefrangemultiformat{equation}{Equations~\textup{#3(#1)#4--#5(#2)#6}}%
{ and \textup{#3(#1)#4--#5(#2)#6}}{, \textup{#3(#1)#4--#5(#2)#6}}{, and \textup{#3(#1)#4--#5(#2)#6}}

\usepackage{tikz}

\usetikzlibrary{arrows}
\usetikzlibrary{patterns}
\usetikzlibrary{calc}
\usetikzlibrary{shapes}
\usetikzlibrary{positioning}

\usepackage[scr=boondox,scrscaled=1.05]{mathalfa}

\renewcommand{\geq}{\geqslant}
\renewcommand{\leq}{\leqslant}

\renewcommand{\succeq}{\succcurlyeq}
\renewcommand{\le}{\leq}
\renewcommand{\ge}{\geq}

\newcommand{\arxiv}[1]{\href{http://arxiv.org/abs/#1}{arXiv:#1}}

\newcommand{\ie}{i.e.}

\DeclareMathOperator{\sign}{\mathsf{sign}}

\DeclareMathAlphabet{\mathbfcal}{OMS}{cmsy}{b}{n}
\DeclareMathAlphabet{\mathbbold}{U}{bbold}{m}{n}

\newcommand{\N}{\mathbb{N}} 
\newcommand{\R}{\mathbb{R}}    
\newcommand{\Q}{\mathbb{Q}}

\newcommand{\puiseux}{\mathbb{K}}

\newcommand{\nnpuiseux}{\mathbb{K}_{\geq 0}}

\newcommand{\trop}[1][]{\ifthenelse{\equal{#1}{}}{ \mathbb{T} }{ \mathbb{T}(#1) }}
\newcommand{\strop}[1][]{{\trop[#1]}_{\pm}}    
\newcommand{\postrop}[1][]{{\trop[#1]}_{+}}    
\newcommand{\negtrop}[1][]{{\trop[#1]}_{-}}

\newcommand{\tplus}{\oplus}  
\newcommand{\tsum}{\bigoplus}
\newcommand{\tdot}{\odot}
\newcommand{\tminus}{\ominus}  
\newcommand{\abs}[1]{|{#1}|}

\newcommand{\zero}{{-\infty}}

\DeclareMathOperator*{\val}{\mathsf{val}}
\DeclareMathOperator*{\lc}{\mathsf{lc}}
\DeclareMathOperator*{\sval}{\mathsf{sval}}

\newcommand{\bo}[1]{\mathbold{#1}}

\newcommand{\card}[1]{|{#1}|}

\newcommand{\spectra}{\mathcal{S}}
\newcommand{\bspectra}{\mathbfcal{S}}
\newcommand{\bsalg}{\mathbfcal{S}}

\newcommand{\homslin}{\slin^{h}}
\newcommand{\rhomslin}{\slin^{rh}}
\newcommand{\fhomspectra}{\spectra^{fh}}

\newcommand{\salg}{\mathscr{S}}

\newcommand{\slin}{\salg}

\newcommand{\lang}{\mathcal{L}}

\newcommand{\logroups}{\lang_{\mathrm{og}}}

\newtheorem{theorem}{Theorem}
\newtheorem{proposition}[theorem]{Proposition}

\newtheorem{corollary}[theorem]{Corollary}
\newtheorem{conjecture}[theorem]{Conjecture}
\newtheorem{lemma}[theorem]{Lemma}
\theoremstyle{definition}
\newtheorem{definition}[theorem]{Definition}
\newtheorem{assumption}{Assumption}
\renewcommand*{\theassumption}{\Alph{assumption}}
\theoremstyle{remark}
\newtheorem{remark}[theorem]{Remark}
\newtheorem{example}[theorem]{Example}

\newlist{assumpenum}{enumerate}{1}
\setlist[assumpenum]{label=(\roman*), ref=\theassumption~(\roman*)}
\crefalias{assumpenumi}{assumption}

\newlist{theoenum}{enumerate}{1}
\setlist[theoenum]{label=(\alph*), ref=\thetheorem~(\alph*)}
\crefalias{theoenumi}{theorem}

\tikzset{grid/.style={gray!30,very thin}}
\tikzset{axis/.style={gray!50,->,>=stealth'}}
\tikzset{convex/.style={draw=none,fill=lightgray,fill opacity=0.7}}
\tikzset{convexborder/.style={very thick}}
\tikzset{point/.style={blue!50}}
\tikzset{hs/.style={fill opacity=0.3,fill=orange,draw=none}}
\tikzset{hsborder/.style={orange,ultra thick,dashdotted}}

\newcommand{\overbar}[1]{\mkern 1.5mu\overline{\mkern-1.5mu#1\mkern-1.5mu}\mkern 1.5mu}

\newcommand{\polyh}{\mathcal{W}}

\newcommand{\bpolyh}{\mathbfcal{W}}

\newcommand{\conv}{\mathrm{conv}}

\newcommand{\psupport}{\Lambda}

\newcommand{\loew}{\succeq}
\newcommand{\inter}{\mathrm{int}}

\newcommand{\vfield}{\mathscr{K}}

\newcommand{\rcfield}{\vfield}

\newcommand{\dgraph}{\vec{\mathcal{G}}}
\newcommand{\vertices}{V}
\newcommand{\vertex}{v}
\newcommand{\edges}{E}
\newcommand{\edge}{e}

\newcommand{\inedge}{\mathrm{In}}
\newcommand{\outedge}{\mathrm{Out}}

\newcommand{\Max}{\mathrm{Max}}
\newcommand{\Min}{\mathrm{Min}}
\newcommand{\Rand}{\mathrm{Rand}}
\newcommand{\Maxvertices}{\vertices_{\Max}}
\newcommand{\Minvertices}{\vertices_{\Min}}
\newcommand{\Randvertices}{\vertices_{\Rand}}

\newcommand{\vertexII}{w}
\newcommand{\vertexIII}{u}
\newcommand{\edgeII}{e'}

\newcommand{\dunion}{\uplus}

\newcommand{\payoff}{r}
\newcommand{\prob}{p}
\newcommand{\probII}{q}

\newcommand{\shapley}{F}
\newcommand{\trshapley}{\shapley'}
\newcommand{\trdgraph}{\dgraph'}

\newcommand{\stbase}{\epsilon}

\newcommand{\smallvar}{\rho}

\newcommand{\tconv}{\mathrm{tconv}}

\newcommand{\transpose}{\intercal}

\title{The tropical analogue of the Helton--Nie conjecture is true}
\date{\today}

\thanks{The three authors were partially supported by the ANR projects CAFEIN (ANR-12-INSE-0007) and MALTHY (ANR-13-INSE-0003), by the PGMO program of EDF and Fondation Math\'ematique Jacques Hadamard, and by
the ``Investissement d'avenir'', r{\'e}f{\'e}rence ANR-11-LABX-0056-LMH,
LabEx LMH. M.~Skomra is  supported by a grant from R{\'e}gion Ile-de-France.}

\author{Xavier {A}llamigeon}
\author{St{\'e}phane {G}aubert}
\author{Mateusz Skomra}
\address{INRIA and CMAP, \'Ecole Polytechnique, CNRS, Universit{\'e} Paris--Saclay, 91128 Palaiseau Cedex France}
\email{firstname.lastname@inria.fr}

\catcode`\@=11
\@namedef{subjclassname@2010}{%
  \textup{2010} Mathematics Subject Classification}
\catcode`\@=12

\begin{document}
\begin{abstract}
Helton and Nie conjectured that
every convex semialgebraic set over the field of real numbers
can be written as the projection of a spectrahedron. 
Recently, Scheiderer disproved this conjecture~ \cite{scheiderer_helton_nie}. We show,
however, that the following result, which may be thought of
as a tropical analogue of this conjecture, is true:
over a real closed nonarchimedean field of Puiseux series,
the convex semialgebraic sets
and the projections of spectrahedra have precisely the same
images by 
the nonarchimedean valuation. 
The proof relies on
game theory methods.
\end{abstract}

\maketitle

\section{Introduction}
Convex semialgebraic sets appear in various guises in computational optimization~\cite{siam_parrilo}. They include spectrahedra, i.e., feasible sets of semidefinite programs (SDPs).
A long-standing problem is to characterize the convex semialgebraic sets that are SDP representable, meaning that they can be represented as the image
of a spectrahedron by a (linear) projector. 
The notion of SDP representability originates from the monograph of Nesterov and Nemirovski~\cite{nesterovnemirovskibook}. Nemirovski asked whether every convex semialgebraic set is SDP representable~\cite{nemirovski_icm}. Helton and Nie
conjectured that the answer is positive.
\begin{conjecture}[{\cite{helton_nie}}]
\label{conj:helton_nie}
Every convex semialgebraic set in $\R^{n}$ is a projection of a spectrahedron.
\end{conjecture}
Several classes of convex semialgebraic sets for which the answer is positive have been identified~\cite{helton_nie,heltonvinnikov,heltonnieMathProg,lasserre,gouveiaparrilothomas,gouveianetzer,eggcurve}.
In particular, it is known
that the conjecture is true in dimension 2~\cite{scheidererdim2}. The conjecture has been recently disproved by Scheidered, who showed that the cone of positive semidefinite forms cannot be expressed as a projection of spectrahedra, except in some particular cases~\cite{scheiderer_helton_nie}. A comprehensive list of references can be found in this work.

\begin{theorem}[\cite{scheiderer_helton_nie}]
The cone of positive semidefinite forms of degree $2d$ in $n$ variables can be expressed as a projection of a spectrahedron only when $2d = 2$ or $n \leq 2$ or $(n, 2d) = (3,4)$.
\end{theorem}  

The notion of convex and semialgebraic sets make sense over any real closed field, in particular over the nonarchimedean field $\puiseux$ of real Puiseux series, equipped with the total order induced by its nonnegative cone $\nnpuiseux$, consisting of series with a nonnegative leading coefficient.
Our main result shows that the next statement, which may be thought of as
a ``Helton--Nie conjecture for valuations,'' is valid.

\begin{theorem}\label{th-main}
The image by the valuation of every convex semialgebraic subset of $\puiseux^n$
coincides with the image by the valuation of a projected spectrahedron over $\puiseux$.
\end{theorem}

Our approach relies on tropical methods. Tropical semialgebraic sets can be defined as the images by the nonarchimedean valuation of semialgebraic sets over $\puiseux$. The quantifier elimination
techniques in real valued fields developed by Pas~\cite{pas_cell_decomposition},
building on work of Denef~\cite{denef_p-adic_semialgebraic}, imply that tropical
semialgebraic sets are semilinear. Moreover, the image by the nonarchimedean valuation of a convex set over $\puiseux$ is a tropical convex set, i.e., a set stable by taking tropical 
convex combinations. In a previous work~\cite{tropical_spectrahedra} we studied tropical spectrahedra, defined as the images by the nonarchimedean valuation of spectrahedra over $\puiseux$, and gave a combinatorial characterization of generic tropical spectrahedra. 
 
The proof relies on the recently developed relations between tropical convex programming
and zero-sum games~\cite{polyhedra_equiv_mean_payoff,issac2016jsc}.
In particular, in the latter reference, we demonstrated a class of generic tropical spectrahedra that corresponds precisely to the sets of subharmonic vectors (subfixed points) of a class of nonlinear Markov operators (Shapley operators of stochastic mean payoff games). In that way, one obtains an explicit construction for these tropical spectrahedra.

The tropical perspective proved to be useful to find counterexamples to classical conjectures in real algebraic geometry. For instance, Itenberg and Viro~\cite{itenberg_viro} disproved the Ragsdale conjecture as an application of the tropical patchworking method.
More recently, Allamigeon et al.~\cite{log_barrier} contradicted,
by a tropical method, the continuous analogue of the Hirsch conjecture
proposed by Deza, Terlaky, and Zinchenko~\cite{deza_terlaky_zinchenko}.
The validity of the tropical analogue of the Helton--Nie conjecture raises the question whether a counterexample could be found by a tropical approach, for instance, by studying images of convex semialgebraic sets and spectrahedra through a map carrying more information than the valuation.

We finally note that semilinear sets that are tropically convex have been studied recently by Bodirsky and Mamino from a different perspective, motivated by a class of satisfiability problems~ \cite{bodirsky_mamino}. They showed in particular that feasibility and infeasibility certificates for these problems can be obtained from stochastic games. The tropical convex sets they consider differ from ours in two respects: the $\zero$ element is not allowed in their approach, whereas it appears as the image of the zero element by the nonarchimedean valuation; moreover, the tropicalizations of convex semialgebraic sets are always closed, and so, definable by weak inequalities, whereas systems including both strict and weak inequalities are considered in~\cite{bodirsky_mamino}.

\section{Preliminaries}

\subsection{Puiseux series and tropical algebra}

A \emph{(generalized formal real) Puiseux series} is a formal series of the form
\begin{equation}
\bo x = \sum_{i = 1}^{\infty} c_{\lambda_{i}}t^{\lambda_{i}} \, , \label{eq:series}
\end{equation}
where $t$ is a formal parameter, $(\lambda_{i})_{i \ge 1}$ is a strictly decreasing sequence of real numbers that is either finite or unbounded, and $c_{\lambda_{i}} \in \R \setminus \{ 0\}$ for all $\lambda_{i}$. There is also a special, empty series, which is denoted by $0$. The Puiseux series can be added and multiplied in the natural way. They form a real closed field~\cite{markwig}, which we denote here by $\puiseux$. If $\bo x$ is a Puiseux series as in~\cref{eq:series}, then by $\lc(\bo x)$ we denote its \emph{leading coefficient}, $\lc(\bo x) = c_{1}$ (with the convention that $\lc(0) = 0$). The (unique) total order on $\puiseux$ is given by the relation $\bo x \ge \bo y \iff \lc(\bo x - \bo y) \ge 0$.

The field of Puiseux series is equipped with a \emph{(nonarchimedean) valuation} $\val \colon \puiseux \to \R \cup \{\zero\}$. If $\bo x \in \puiseux$ is as in~\cref{eq:series}, then we define $\val(\bo x)$ as the leading exponent of $\bo x$, $\val(\bo x) = \lambda_{1}$ (with the convention that $\val(0) = \zero$). 
It follows from the definition that we have the relations
\begin{align}
\val(\bo x + \bo y) &\leq \max( \val (\bo x ), \val (\bo y))\label{e-val}\\
\val(\bo x  \bo y) &= \val (\bo x ) + \val (\bo y)\label{e-val2}\\
0 \le \bo x \le \bo y &\implies \val(\bo x) \le \val(\bo y)\label{e-val3}
\end{align}
Furthermore, the inequality in~\cref{e-val} becomes an equality when the leading terms of $\bo x$ and $\bo y$ do not cancel, which is the case if $\val (\bo x ) \neq \val (\bo y)$ or if
$\bo x, \bo y\geq 0$. We denote by $\nnpuiseux$ the set of nonnegative Puiseux series (the series which fulfill the inequality $\bo x \ge 0$). 

\begin{remark}
We chose the specific field $\puiseux$ for simplicity of exposition. Actually, a quantifier elimination argument allows one to deduce that our main results stated over $\puiseux$ are also valid over other real closed nonarchimedean fields, see \cref{rk-anyvaluedfield}.
\end{remark}

\subsection{Tropical semifield}

The \emph{tropical semifield} $\trop$ describes the algebraic structure of $\puiseux$ under the valuation map. The underlying set of $\trop$ is defined as $\trop \coloneqq \R \cup \{\zero\}$, the tropical addition is defined as $x \tplus y = \max(x,y)$, and the tropical multiplication is defined as $x \tdot y = x + y$. The structure $\trop$ constitutes only a semifield, for the tropical addition does not have an inverse operation. 
We use the notation $\tsum_{i = 1}^{n}a_{i} = a_{1} \tplus \cdots \tplus a_{n}$ and $a^{\tdot n} = a \tdot \cdots \tdot a$ ($n$ times). We also endow $\trop$ with the standard order $\ge$. The properties \cref{e-val,e-val2,e-val3} imply that $\val$ is an order-preserving morphism of semifields from $\nnpuiseux$ to $\trop$. 

When dealing with semialgebraic sets, it is convenient to keep track not only of the valuations of the elements of $\puiseux$, but also of their signs. To this end, we introduce the sign function $\sign \colon \puiseux \to \{-1,0,+1\}$ defined as $\sign(\bo x) = 1$ if $\bo x > 0$, $\sign(\bo x) = -1$ if $\bo x < 0$, and $\sign(0) = 0$. The set of \emph{signed tropical numbers} is then defined as $\strop \coloneqq (\{+1,-1 \} \times \R) \cup \{(0,\zero) \}$, and the \emph{signed valuation} is defined as $\sval \colon \puiseux \to \strop$, $\sval(\bo x) = (\sign(\bo x), \val(\bo x))$. The \emph{modulus} function $\abs{\cdot} \colon \strop \to \trop$ is defined as the projection which forgets the first coordinate. The \emph{sign} function $\sign \colon \strop \to \{-1,0,1\}$ is defined as the projection which forgets the second coordinate. The elements of the form $(1,a)$ of $\strop$ are called \emph{positive tropical numbers} and are denoted by $\postrop$. Similarly, the elements of the form $(-1,a)$ of $\strop$ are called \emph{negative tropical numbers} and are denoted by $\negtrop$. By convention, we denote the positive tropical number $(1,a)$ by $a$, the negative tropical number $(-1,a)$ by $\tminus a$, and the element $(0,\zero)$ by $\zero$. Here, $\tminus$ is a formal symbol. Note that we can extend the definition of tropical multiplication to $\strop$ using the usual rules for signs (e.g., we have $(\tminus 3)  \tdot 7 = \tminus 10$ and $(\tminus 3)  \tdot (\tminus 7) = 10$). However, we only partially extend the tropical addition to the elements of $\strop$ which have the same sign (e.g., we have $3 \tplus 7 = 7$ and $(\tminus 3)  \tplus (\tminus 7) = \tminus 7$, but $(\tminus 3)  \tplus 7$ is not defined). 
One can extend the set $\strop$ even further
to get a semiring with a well-defined tropical addition~\cite{guterman}, or work with hyperfields~\cite{virohyperfields,connesconsani} instead of semifields, 
but we do not need that here. Furthermore, note that the tropical semiring $\trop$ is isomorphic to $\postrop \cup \{ \zero\}$.

A \emph{tropical (signed) polynomial} over the variables $X_1, \dots, X_n$ is a formal expression of the form
\begin{align}
P(X) = \tsum_{\alpha \in \psupport} a_{\alpha} \tdot X_{1}^{\tdot \alpha_{1}} \tdot  \cdots \tdot X_{n}^{\tdot \alpha_{n}} \, ,\label{e-def-P}
\end{align}
where $\psupport \subset \{0, 1, 2, \dots \}^{n}$, and $a_{\alpha} \in \strop \setminus \{\zero\}$ for all $\alpha \in \psupport$. 
If $P$ is given as in~\cref{e-def-P}, we define $P^+$ (resp.\ $P^-$) as the tropical polynomial generated by the terms $\abs{a_\alpha} \tdot X_{1}^{\tdot \alpha_{1}} \tdot  \dots \tdot X_{n}^{\tdot \alpha_{n}}$ where $a_\alpha \in \postrop$ (resp.\ $\negtrop$). Note that the quantities $P^+(x)$ and $P^-(x)$ are well defined for all $x \in \trop^n$, because the tropical polynomials $P^+$ and $P^-$ have only tropically positive coefficients.

We extend the functions $\val$ and $\sval$ to vectors and matrices by applying them coordinatewise. 

\subsection{Tropical convexity}
\label{sec:tropical_convexity}

In this section, we recall some basic facts about convexity in the usual and tropical sense. A set $\bo X \subset \puiseux^{n}$ is called \emph{convex} if for every $\bo x, \bo y \in \bo X$ and every $\bo \lambda \in \puiseux$ such that $0 \le \bo \lambda \le 1$ we have $\bo \lambda + (1 - \bo \lambda)\bo y \in \bo X$. Since the intersection of any number of convex sets in convex, for every set $\bo X \subset \puiseux^{n}$ we can define its \emph{convex hull} (denoted $\conv(\bo X)$) as the smallest (inclusionwise) convex set that contains $\bo X$. This set is characterized by Carath{\'e}odory's theorem.

\begin{theorem}\label{caratheodory}
If $\bo X \subset \puiseux^{n}$, then we have the equality
\[
\conv(\bo X) = \Bigl\{\sum_{k = 1}^{n+1} \bo \lambda_{k} \bo x_{k} \in \puiseux^{n} \colon \forall k, \, \bo x_{k} \in \bo X \, \land \, \forall k, \, \bo \lambda \ge 0 \, \land \, \sum_{k = 1}^{n+1} \bo \lambda_{k} = 1 \Bigr\} \, .
\]
\end{theorem}

We refer to~\cite[Corollary~7.1j]{schrijver} for a proof of \cref{caratheodory} that is valid over every ordered field. 

Let us now move to tropical convexity, referring the reader to~\cite{cgq02,tropical_convexity} for background. We say that a set $X \subset \trop^{n}$ is \emph{tropically convex} if for every $x, y \in X$ and every $\lambda, \mu \in \trop$ such that $\lambda \tplus \mu = 0$ the point $(\lambda \tdot x) \tplus (\mu \tdot y)$ belongs to $X$. The latter quantity corresponds to the tropical analogue of a convex combination of $x$ and $y$. Indeed, the scalars $\lambda$ and $\mu$ are implicitly ``nonnegative'' in the tropical sense, as they are greater than or equal to the tropical zero element $-\infty$. Besides, their tropical sum equals the tropical unit $0$.
The intersection of any number of tropically convex sets is tropically convex. Hence, for any set $X \subset \trop^{n}$ we can define its \emph{tropical convex hull} (denoted $\tconv(X)$) as the smallest (inclusionwise) tropically convex set that contains $X$. Alternatively, one may work with a \emph{tropical (convex) cone} $X$, defined by requiring $(\lambda \tdot x) \tplus (\mu \tdot y) \in X$
for all $\lambda,\mu\in \trop$. 
Tropical convex sets can be identified to cross sections of tropical
convex cones~\cite{cgq02}. 
Carath{\'e}odory's theorem is still true in the tropical setting:

\begin{theorem}[{\cite{helbig}, \cite{BriecHorvath04}, \cite{tropical_convexity}}]\label{tropical_caratheodory}
If $X \subset \trop^{n}$, then we have the equality
\[
\tconv(X) = \Bigl\{ \tsum_{k = 1}^{n+1}(\lambda_{k} \tdot x_{k}) \colon \forall k, \,  x_{k} \in X \, \land \,\tsum_{k = 1}^{n+1} \lambda_{k} = 0 \Bigr\} \, .
\]
\end{theorem}

A relation between the convexity in $\puiseux$ and the tropical convexity is shown in the next two lemmas.

\begin{lemma}\label{valuation_of_convex_set}
If $\bo X \subset \puiseux^{n}$ is a convex set, then $\val(\bo X)$ is tropically convex.
\end{lemma}
\begin{proof}
Let $x, y \in \val(\bo X)$ and take any $\lambda, \mu \in \trop$ such that $\lambda \tplus \mu = 0$. Without loss of generality, suppose that $\lambda = 0$. Take any points $\bo x \in \bo X \cap \val^{-1}(x)$ and $\bo y \in \bo X \cap \val^{-1}(y)$. Let us look at two cases. If $\mu < 0$, 
then for any real positive constant $c$, we have $1-ct^{\mu}>0$, and
so the point $\bo z = (1 - ct^{\mu})\bo x + ct^{\mu}\bo y$ belongs to $\bo X$. 
We already noted that the equality holds in the inequality~\cref{e-val}
if the leading terms do not cancel. Hence, choosing 
$c$ such that 
$c \neq -\lc(\bo x_{k})/\lc(\bo y_{k})$
for all $k \in [n]$ satisfying $\bo y_{k} \neq 0$, we deduce that
$\val(\bo z) = (\lambda \tdot x) \tplus (\mu \tdot y)$. 
If $\mu  = 0$, then we take now a real constant $c \in (0,1)$ such that for all $k \in [n]$ satisfying $\bo y_{k} \neq 0$ we have $c/(1 - c) \neq -\lc(\bo x_{k})/\lc(\bo y_{k})$. Then, the point $\bo z = (1 - c)\bo x + c\bo y$ belongs to $\bo X$ and we deduce as above that $\val(\bo z) = (\lambda \tdot x) \tplus (\mu \tdot y)$.
\end{proof}

The next lemma shows that a tighter relation holds for sets included in the nonnegative orthant of $\puiseux$.

\begin{lemma}\label{valuation_of_convex_hull}
If $\bo X \subset \nnpuiseux^{n}$ is any set, then we have $\val(\conv(\bo X)) = \tconv(\val(\bo X))$.
\end{lemma}
\begin{proof}
We start by proving the inclusion $\subset$. Take a point $\bo y \in \conv(\bo X)$. By \cref{caratheodory}, there exist $\bo \lambda_{1}, \dots, \bo \lambda_{n+1} \ge 0$ and $\bo x_{1}, \dots, \bo x_{n+1} \in \bo X$ such that $\bo y = \bo \lambda_{1} \bo x_{1} + \dots + \bo \lambda_{n+1} \bo x_{n+1}$. Hence, by \cref{e-val,e-val2} (and using the fact that $\bo X \subset \nnpuiseux^{n}$) we have 
\[
\val(\bo y) = \bigl(\val(\bo \lambda_{1}) \tdot \val(\bo x_{1}) \bigr) \tplus \dots \tplus \bigl(\val(\bo \lambda_{n+1}) \tdot \val(\bo x_{n+1}) \bigr) \, .
\]
Furthermore, we have $\sum_{k = 1}^{n+1} \bo \lambda_{k} = 1$ and hence $\tsum_{k = 1}^{n+1} \val(\bo \lambda_{k}) = 0$. Therefore, $\val(\bo y) \in \tconv(X)$ by \cref{tropical_caratheodory}. Conversely, take any point $y \in \tconv(X)$. By \cref{tropical_caratheodory}, we can find $\lambda_{1}, \dots, \lambda_{n+1} \in \trop$, $\tsum_{k = 1}^{n+1} \lambda_{k} = 0$ and $x_{1}, \dots, x_{n+1} \in X$ such that $y = (\lambda_{1} \tdot x_{1}) \tplus \dots (\lambda_{n+1} \tdot x_{n+1})$. We define $\bo \lambda_k \coloneqq t^{\lambda_k} / (\sum_{l = 1}^n t^{\lambda_l})$. Observe that for all $k$, $\val(\bo \lambda_k) = \lambda_k$ because the term $\sum_{l = 1}^n t^{\lambda_l}$ has valuation $\tsum_{l = 1}^{n+1} \lambda_l = 0$. Moreover, we have $\bo \lambda_{k} \ge 0$ and $\sum_{k = 1}^{n+1} \bo \lambda_{k} = 1$. Hence, the point $\bo y = \bo \lambda_{1} \bo x_{1} + \dots + \bo \lambda_{n+1} \bo x_{n+1}$ belongs to $\conv(\bo X)$ and verifies $\val(\bo y) = y$.
\end{proof}

We also need the following lemma.

\begin{lemma}\label{convex_hull_of_union}
Suppose that sets $X, Y \subset \trop^{n}$ are tropically convex. Then we have the equality
\[
\tconv(X \cup Y) = \{(\lambda \tdot x) \tplus (\mu \tdot y) \in \trop^{n} \colon x \in X, y \in Y, \lambda \tplus \mu  = 0 \} \, .
\]
\end{lemma}
\begin{proof}
The inclusion $\supset$ follows immediately from the definition of tropical convex hull. The other inclusion holds because the set on the right-hand side contains $X$ and $Y$ and is tropically convex.
\end{proof}

\subsection{Tropicalization of convex semialgebraic sets}\label{sec:tropicalization_of_convex}

A set $\bsalg \subset \puiseux^{n}$ is called \emph{basic semialgebraic} if it is of the form
\begin{equation}
\{ \bo x \in \puiseux^{n} \colon \forall i = 1,\dots,p, \, \bo P_{i}(\bo x) > 0 \land \forall i = p+1,\dots,q, \, \bo P_{i}(\bo x) = 0 \} \, , \label{eq:basic}
\end{equation}
where $\bo P_{i} \in \puiseux[X_{1}, \dots, X_{n}]$ are polynomials. A set $\bsalg \subset \puiseux^{n}$ is called \emph{semialgebraic} if it is a finite union of basic semialgebraic sets. In this section, we characterize the sets that arise as images by valuation of convex semialgebraic sets.

\begin{lemma}\label{convex_hull_salg}
If $\bsalg \subset \puiseux^{n}$ is a semialgebraic set, then $\conv(\bsalg)$ is also semialgebraic.
\end{lemma}
\begin{proof}
This is an immediate consequence of \cref{caratheodory} and of the quantifier elimination in real closed fields~\cite[Theorem~3.3.15]{marker_model_theory}.
\end{proof}

Let us make the following definition.

\begin{definition}
We say that a set $\salg \subset \trop^{n}$ is a \emph{tropicalization of a convex semialgebraic} set if there exists a convex semialgebraic set $\bsalg \subset \puiseux^{n}$ such that $\val(\bsalg) = \salg$.
\end{definition}

Given $d \geq 1$, the \emph{support} of a point $y \in \trop^d$ is defined as the set of indices $k \in [d]$ such that $y_k \neq \zero$. Given a nonempty subset $K \subset [d]$, and a set $Y \subset \trop^d$, we define the \emph{stratum of $Y$ associated with $K$} as the subset of $\R^K$ formed by the projection $(y_k)_{k \in K}$ of the points $y \in Y$ with support~$K$. The stratum associated with the set $[d]$ is referred to as the \emph{main stratum}.

We say that a set $\salg \subset \R^{d}$ is a \emph{basic semilinear} set if it is a relatively open polyhedron of the form
\[
\{x \in \R^{d} \colon \forall i = 1,\dots,p, \, \langle A_{i}, x \rangle > b_{i} \land \forall i = p+1,\dots,q, \, \langle A_{i}, x \rangle = b_{i} \} \, ,
\]
where the matrix $A \in \Q^{q \times d}$ is rational, the vector $b \in \R^{q}$ is real, and $\langle \cdot, \cdot \rangle$ denotes the standard scalar product in $\R^{d}$. We say that a set is \emph{semilinear} if it is a finite union of basic semilinear sets. Note that $\salg \subset \R^{d}$ is a closed semilinear set if and only if it is a finite union of polyhedra of the form $Ax \ge b$, where the matrix $A \in \Q^{q \times d}$ is rational and the vector $b \in \R^{q}$ is real. The following proposition characterizes the tropicalizations of convex semialgebraic sets. This result is based on the Denef--Pas quantifier elimination in the theory on real closed valued fields.
\begin{proposition}\label{tropical_conv_salg}
A set $\salg \subset \trop^{n}$ is a tropicalization of a 
convex semialgebraic set if and only if $\salg$ is tropically convex and every stratum of $\salg$ is a closed semilinear set.
\end{proposition}
\begin{proof}
The ``only if'' part follows from \cite[Theorems~4 and~10]{tropical_spectrahedra} and \cref{valuation_of_convex_set}. To prove the opposite implication, suppose that $\salg$ is tropically convex and has closed semilinear strata. Therefore, it is a finite union of sets of the form $\polyh = \{ x \in \trop^{n} \colon Ax_{K} \ge b, \, x_{[n]\setminus K} = \zero \}$, where, for every $L \subset [n]$, $x_{L}$ denotes the vector formed by the coordinates of $x$ taken from $L$, the matrix $A \in \Q^{m \times \card{K}}$ is rational and the vector $b \in \R^{m}$ is real. Take any such set $\polyh$ and consider the set
\[
\bpolyh \coloneqq \{\bo x \in \nnpuiseux^{n} \colon \forall i \in [m], \, \prod_{k \in K} \bo x_{k}^{A_{ik}}\ge t^{b_{i}} \land \bo x_{[n] \setminus K} = 0 \} \, .
\]
Note that $\bpolyh$ is semialgebraic. Moreover, by \cref{e-val,e-val2,e-val3} we have $\val(\bpolyh) \subset \polyh$. Furthermore, if $x \in \polyh$, and we take $\bo x_{k} \coloneqq t^{x_{k}}$ for all $k \in [n]$ (with the convention that $t^{\zero} = 0$), then we have $\bo x \in \bpolyh$. Therefore $\val(\bpolyh) = \polyh$. Let $\bo U$ denote the union of all sets $\bpolyh$ that arise in this way. We have $\val(\bo U) = \salg$ and $\bo U$ is semialgebraic. Thus, if we take $\bsalg \coloneqq \conv(\bo U)$, then $\bsalg$ is convex and semialgebraic by \cref{convex_hull_salg}. Moreover, \cref{valuation_of_convex_hull} shows that $\val(\bsalg) = \salg$.
\end{proof}

\subsection{Tropical Metzler spectrahedra}

Let us recall that a real symmetric matrix is positive semidefinite if it admits a Cholesky decomposition. This is equivalent to the nonnegativity of its principal minors, its smallest eigenvalue, and the associated quadratic form. All of these properties are still equivalent for symmetric matrices defined over arbitrary real closed fields, such as Puiseux series (this is a consequence of the completeness of the theory of such fields, see~\cite[Corollary~3.3.16]{marker_model_theory}). This implies that the definition of a spectrahedron is valid over $\puiseux$.

\begin{definition}
Suppose that $\bo Q^{(0)}, \dots, \bo Q^{(n)} \in \puiseux^{m \times m}$ are symmetric matrices. Then, the \emph{spectrahedron} associated with these matrices is defined as
\[
\bspectra = \{\bo x \in \puiseux^{n} \colon \bo Q^{(0)} + \bo x_{1} \bo Q^{(1)} + \dots + \bo x_{n} \bo Q^{(n)} \loew 0 \} \, ,
\]
where the symbol $\loew$ denotes the Loewner order on symmetric matrices. (By definition, $X \loew Y$ if $X - Y$ is positive semidefinite.)
\end{definition}

In our previous works \cite{tropical_spectrahedra,issac2016jsc} we introduced the notion of tropical spectrahedra and a special subclass of these objects called tropical Metzler spectrahedra. The latter have a simpler combinatorial description; moreover, any generic tropical spectrahedron can be represented by a boolean combination of tropical Metzler spectrahedra~\cite[Sections~5.3--5.4]{tropical_spectrahedra}. Tropical spectrahedra are defined as follows.

\begin{definition}\label{definition:tropical_spectra}
We say that a set $\spectra \subset \trop^{n}$ is a \emph{tropical spectrahedron} if there exists a spectrahedron $\bspectra \subset \nnpuiseux^{n}$ such that $\spectra = \val(\bspectra)$.
\end{definition}

A square matrix $\bo M$ is called a \emph{(negated) Metzler matrix} if its off-diagonal entries are nonpositive. 
Similarly, a matrix $M \in \strop^{m \times m}$ is called a \emph{tropical Metzler matrix} if its off-diagonal entries belong to $\negtrop \cup \{\zero \}$. Fix a sequence of symmetric tropical Metzler matrices $Q^{(0)}, \dots, Q^{(n)} \in \strop^{m \times m}$. For every pair $(i,j) \in [m]^{2}$ we consider the tropical polynomial $Q_{ij}(X)$ defined as
\[
Q_{ij}(X) \coloneqq Q^{(0)}_{ij} \tplus (Q^{(1)}_{ij} \tdot X_{1}) \tplus \dots \tplus (Q^{(n)}_{ij} \tdot X_{n}) \, .
\]

\begin{definition}\label{definition:tropical_metzler_spectra}
The \emph{tropical Metzler spectrahedron} described by $Q^{(0)}, \dots, Q^{(n)} \in \strop^{m \times m}$, denoted $\spectra(Q^{(0)} | Q^{(1)},\dots, Q^{(n)})$, is the set of all points $x \in \trop^{n}$ which satisfy the following two conditions:
\begin{itemize}[nosep]
\item $Q_{ii}^{+}(x) \ge Q_{ii}^{-}(x)$ for every $i \in [m]$
\item $Q_{ii}^{+}(x) \tdot Q_{jj}^{+}(x) \ge (Q_{ij}(x))^{\tdot 2}$ for every $i,j \in [m]^{2}, i \neq j$.
\end{itemize}
Note that the function $Q_{ij}(x)$ is well-defined for all $x \in \trop^{n}$ 
because every
$Q^{(k)}$ is a tropical Metzler matrix. If the matrix $Q^{(0)}$ is equal to $\zero$, then we say that $\spectra(\zero|Q^{(1)},\dots, Q^{(n)})$ is a \emph{tropical Metzler spectrahedral cone}. We say that $\spectra(Q^{(0)} | Q^{(1)},\dots, Q^{(n)})$ is \emph{real} if it is included in $\R^{n}$. 
\end{definition}

\begin{remark}
The definition above differs slightly from the one of~\cite{tropical_spectrahedra,issac2016jsc}. 
Indeed, 
in these references
it was enough to work with tropical Metzler spectrahedral cones, while the use of affine tropical spectrahedra is indispensable in the context of the Helton--Nie conjecture. The connection between the two notions is given in \cite[Lemma~20]{tropical_spectrahedra}.
\end{remark}

The name ``tropical Metzler spectrahedron'' is justified by the fact that these sets are indeed tropical spectrahedra, as shown in \cite[Proposition~23 and Lemma~20]{tropical_spectrahedra}.

\begin{proposition}\label{metzler_is_spectra}
Every tropical Metzler spectrahedron is a tropical spectrahedron.
\end{proposition}

\section{The tropical analogue of the Helton--Nie conjecture}
\label{sec:tropical_consequence}

As stated in the introduction, Scheiderer has shown 
that the cone of positive semidefinite forms over $\R$ is a counterexample to the Helton--Nie conjecture~\cite{scheiderer_helton_nie}. We first note that this yields a counterexample to the analogue of this conjecture over Puiseux series.

\begin{corollary}[of~{\cite[Corollary~4.25]{scheiderer_helton_nie}}]
The cone of positive semidefinite forms of degree $2d$ in $n$ variables over $\puiseux$ can be expressed as a projection of a spectrahedron over $\puiseux$ only when $2d = 2$ or $n \leq 2$ or $(n, 2d) = (3,4)$.
\end{corollary}

\begin{proof}
Consider a real closed field $\rcfield$, and integers $d$, $m$, $n$, and $p$. The statement ``the cone of positive semidefinite forms of degree $2d$ in $n$ variables over $\rcfield$ is the projection of a spectrahedron in $\rcfield^p$ associated with matrices of size $m \times m$'' is a sentence in the language of ordered rings. Since the theory of real closed fields is complete~\cite[Corollary~3.3.16]{marker_model_theory}, this sentence is true over $\R$ if and only if it is true over $\rcfield$.  
\end{proof}

We next state the main result of this paper. 
We shall prove a special case of this result in \cref{section:realcones}, and derive the general case in \cref{section:general}. 
\begin{theorem}\label{theorem:tropical_convex_sets}
Fix a set $\slin \subset \trop^{n}$. Then, the following conditions are equivalent:
\begin{theoenum}
\item $\slin$ is a tropicalization of a convex semialgebraic set \label{main:item1}
\item $\slin$ is tropically convex and has closed semilinear strata \label{main:item2}
\item $\slin$ is tropically convex and every stratum of $\slin$ is a projection of a real tropical Metzler spectrahedron \label{main:item3}
\item $\slin$ is a projection of a tropical Metzler spectrahedron. \label{main:item4}
\end{theoenum}
\end{theorem}

We point out that \cref{th-main} is a corollary of \cref{theorem:tropical_convex_sets}.

\begin{proof}[Proof of \cref{th-main}]
Let $\bsalg \subset \puiseux^{n}$ be any convex semialgebraic set. By \cref{theorem:tropical_convex_sets}, the set $\val(\bsalg) \subset \trop^{n}$ is a projection of a tropical Metzler spectrahedron $\spectra \subset \trop^{n + n'}$. By \cref{metzler_is_spectra}, there is a spectrahedron $\bspectra' \subset \nnpuiseux^{n + n'}$ such that $\val(\bspectra') = \spectra$. Let $\bo \pi \colon \puiseux^{n+n'} \to \puiseux^{n}$ denote the projection on the first $n$ coordinates. Then $\slin = \val(\bo \pi(\bspectra'))$ and hence $\val(\bsalg) =  \val(\bo \pi(\bspectra'))$.
\end{proof}

\section{Tropical Helton--Nie conjecture for real tropical cones}\label{section:realcones}

In this section, we show that the tropical analogue of Helton--Nie conjecture is true for real tropical cones. 
We say that a set $X \subset \R^{n}$ is a \emph{real tropical cone} if for every $x, y \in X$ and every $\lambda, \mu \in \R$ we have $(\lambda \tdot x) \tplus (\mu \tdot y) \in X$. A real tropical cone is nothing but the main stratum of a tropical cone as defined in \cref{sec:tropical_convexity}.
Indeed, if $Y$ is a tropical cone, then $Y\cap \R^n$ is a real tropical cone, whereas if $X$ is a real tropical cone, then $X\cup\{-\infty\}$ is a tropical cone. 

\subsection{Preliminaries on semilinear monotone homogeneous operators}

We say that a function $\shapley \colon \R^{n} \to \R^{m}$ is \emph{piecewise affine} if there exists a set of full-dimensional polyhedra $\polyh^{(1)}, \dots, \polyh^{(p)} \subset \R^{n}$ satisfying $\bigcup_{s = 1}^{p} \polyh^{(s)} = \R^{n}$ and such that the restriction of $\shapley$ to $\polyh^{(s)}$ is affine, i.e., $\shapley_{|\polyh^{(s)}}(x) = A^{(s)}x + b^{(s)}$ for some matrix $A^{(s)} \in \R^{m \times n}$ and vector $b^{(s)} \in \R^{m}$. In particular, piecewise affine functions are continuous (since the polyhedra $\polyh^{(1)},\dots, \polyh^{(p)}$ are closed). We shall say that the family $(\polyh^{(s)}, A^{(s)}, b^{(s)})_s$ is a \emph{piecewise description} of the function $\shapley$. 

We recall the following minimax representation result proved by Ovchinnikov~\cite{ovchinnikov}, in which we denote $\shapley(x) = (\shapley_{1}(x), \dots, \shapley_{m}(x))$.
\begin{theorem}[\cite{ovchinnikov}]\label{theorem:min_max_representation}
Suppose that the function $\shapley \colon \R^{n} \to \R^{m}$ is piecewise affine, and let $(\polyh^{(s)}, A^{(s)}, b^{(s)})_{s \in [p]}$ be a piecewise description of $\shapley$. Then, for every $k \in [n]$ there exists a number $M_{k} \ge 1$ and a family $\{ S_{k i}\}_{i \in [M_{k}]}$ of subsets of~$[p]$ such that for all $x \in \R^{n}$ we have
\[
\shapley_{k}(x) = \min_{i \in [M_{k}]} \max_{s \in S_{ki}} (A^{(s)}_{k}x + b^{(s)}_{k}) \, .
\]
\end{theorem}

We say that a function $\shapley \colon \R^{n} \to \R^{m}$ is \emph{semilinear} if its graph $\{(x,y) \in \R^{n \times m} \colon y = \shapley(x) \}$ is a semilinear set. The next lemma shows that continuous semilinear functions are piecewise affine.

\begin{lemma}\label{lemma:piecewise_affine}
Suppose that the continuous function $\shapley \colon \R^{n} \to \R^{m}$ is semilinear. Then, it is piecewise affine. Moreover, it admits a piecewise description $(\polyh^{(s)}, A^{(s)}, b^{(s)})_{s \in [p]}$ such that the polyhedra $\polyh^{(s)}$ are semilinear, and the matrices $A^{(s)}$ are rational. 
\end{lemma}
\begin{proof}
Since $\shapley$ is continuous and semilinear, the graph of $\shapley$ is a closed semilinear set. Therefore, it is a finite union of semilinear polyhedra. Let $\{ (x,y) \colon Bx + Cy \ge d \}$, where $B \in \Q^{p \times n}$, $C \in \Q^{p \times m}$, $d \in \R^{p}$ be one of these polyhedra. If we fix $\overbar{x}$, then, by the definition of a graph, the polyhedron consisting of all $y$ such that $Cy \ge d - B\overbar{x}$ reduces to a point $\overbar{y}$. Thus, there exists an invertible submatrix $C_{I} \in \Q^{m \times m}$ of $C$ such that $\overbar{y} = C_{I}^{-1}(d_{I} - B_{I}\overbar{x}) = C_{I}^{-1}d_{I} - C_{I}^{-1}B_{I}\overbar{x}$. In other words, the graph of $\shapley$ is a finite union of polyhedra of the form 
\[
\polyh = \{(x,y) \colon Bx + Cy \ge d, y = C_{I}^{-1}d_{I} - C_{I}^{-1}B_{I}x \} \, ,
\] 
where $C_{I}$ is an invertible submatrix of $C$. As a result, if $\pi \colon \R^{n + m} \to \R^{n}$ denotes the projection on the first $n$ coordinates, and $x \in \pi(\polyh)$ is any point, then we have $\shapley(x) = C_{I}^{-1}d_{I} - C_{I}^{-1}B_{I}x$. By eliminating the polyhedra $\pi(\polyh)$ that are not full dimensional, we obtain a piecewise description of $\shapley$ satisfying the expected requirements.
\end{proof}

We say that a selfmap $\shapley \colon \R^n \to \R^n$ is \emph{monotone} if $\shapley(x) \leq \shapley(y)$ as soon as $x \leq y$, where $\leq$ denotes the coordinatewise partial order over $\R^n$. Such a function is said to be \emph{(additively) homogeneous} if $\shapley(\lambda + x) = \lambda + \shapley(x)$ for all $\lambda \in \R$ and $x \in \R^n$. Here, if $z \in \R^n$, then $\lambda + z$ stands for the vector with entries~$\lambda + z_k$. 

The following observation is well known~\cite{crandalltartar}.
\begin{lemma}\label{lemma:nonexpansive}
Every monotone homogeneous operator is nonexpansive in the supremum norm.
\end{lemma}

\begin{proof}
Observe that $x \leq \| x-y\| + y$. Therefore, we get $\shapley(x) \leq \shapley(\| x-y\| + y) = \| x-y\| + \shapley(y)$. 
\end{proof}

Kolokoltsov showed that every monotone homogeneous operator $\shapley$ has a minimax representation as a dynamic programming operator of a zero-sum game~~\cite{kolokoltsov}.
When $\shapley$ is semilinear, the following result shows that we have a finite representation
of the same nature.
\begin{lemma}\label{lemma:semilinear_min_max}
If $\shapley \colon \R^{n} \to \R^{n}$ is semilinear, monotone, and homogeneous, then it can be written in the form
\begin{align}\label{eq:semilineargames}
\forall k, \, \shapley_{k}(x) = \min_{i \in [M_{k}]} \max_{s \in S_{ki}} (A^{(s)}_{k}x + b^{(s)}_{k}) \, ,
\end{align}
where $A^{(1)}, \dots, A^{(p)} \in \Q^{n \times n}$ is a sequence of stochastic matrices, $b^{(s)} \in \R^{n}$ for all $s \in [n]$,  $M_{k} \ge 1$ for all $k \in [n]$, and $S_{ki}$ is a subset of~$[p]$ for every $k \in [n]$ and $i \in [M_{k}]$.
\end{lemma}
\begin{proof}
\Cref{lemma:nonexpansive} shows that $\shapley$ is continuous. Let $(\polyh^{(s)}, A^{(s)}, b^{(s)})_{s \in [p]}$ a piecewise description of $\shapley$ as provided by \cref{lemma:piecewise_affine}. In particular, every matrix $A^{(s)}$ is rational. We want to show that it is stochastic. To this end, take any $x \in \inter(\polyh_{s})$. Let $y$ be the sum of the columns of $A^{(s)}$. Since $\shapley$ is homogeneous, for any $\smallvar > 0$ small enough we have $\shapley(\smallvar + x ) = A^{(s)}x + b^{(s)} + \smallvar y = \smallvar + \shapley(x)$. In other words, the sum of every line of $A^{(s)}$ is equal to $1$. Let $\stbase_{k}$ denote the $k$th vector of standard basis in $\R^{n}$. Since $\shapley$ is monotone, for $\smallvar > 0$ small enough we have $\shapley(x + \smallvar \stbase_{k}) = A^{(s)}x + b^{(s)} + \smallvar A^{(s)}\stbase_{k} \ge \shapley(x)$. In other words, the matrix $A^{(s)}$ has nonnegative entries in its $k$th column. Since $k$ was arbitrary, $A^{(s)}$ is stochastic. Therefore, the claim follows from \cref{theorem:min_max_representation}.
\end{proof}

We now characterize the class of closed semilinear real tropical cones. To this end, we use the model-theoretic definition of semilinear sets. Let $\logroups \coloneqq (0, +, \le)$ denote the language of ordered groups. Then, the elimination of quantifiers in divisible ordered abelian groups \cite[Theorem~3.1.17]{marker_model_theory}, shows that a set $\salg \subset \R^{n}$ is semilinear if and only if there exists a number $m \ge 0$, an $\logroups$-formula $\psi(x_{1}, \dots, x_{n+m})$, and a vector $\overbar{b} \in \R^{m}$ such that
\[
\salg = \{x \in \R^{n} \colon \psi(x_{1}, \dots, x_{n}, \overbar{b}) \; \text{is true in}\; \R \} \, .
\]

\begin{proposition}\label{tropical_cones_from_operators}
A set $\slin \subset \R^{n}$ is a closed semilinear real tropical cone if and only if there exists a semilinear monotone homogeneous operator $\shapley \colon \R^{n} \to \R^{n}$ such that $\slin = \{x \in \R^{n} \colon x \le \shapley(x) \}$.
\end{proposition}
\begin{proof}
To prove the first implication, we consider two cases. If $\slin$ is empty, then we take $\shapley(x) = x - (1,\dots,1)$. Otherwise, we define $\shapley$ by $\shapley_k(x) \coloneqq \sup \{y_k \colon y \in \slin \, , \ y \le x \}$ for all $k \in [n]$. We claim that every supremum is attained. Indeed, the set $\{y \in \slin \colon y \leq x \}$ is nonempty (take an arbitrary $z \in \slin$, and consider $\lambda + z$ for $\lambda \in \R$ small enough), closed, and bounded by $x$. Let $y^{(k)} \in \slin$ attaining the maximum in $\shapley_k(x)$. Then the point $y^{(1)} \tplus \dots \tplus y^{(n)}$ is an element of $\slin$ smaller than or equal to $x$. We deduce that it coincides with $\shapley(x)$. Subsequently, $\shapley(x)$ belongs to $\slin$. 

The operator $\shapley$ is semilinear because the supremum is definable in the language~$\logroups$, and $\slin$ is semilinear. Besides, $\shapley$ is obviously monotone. It is also homogeneous because if $y \in \slin$, then $\lambda + y \in \slin$ for all $\lambda \in \R$. Finally, the inclusion $\slin \subset \{ x \in \R^n \colon x \leq \shapley(x) \}$ is straightforward, while the inverse inclusion follows from the fact that if $x \leq \shapley(x)$, then $x = \shapley(x)$. 

Conversely, fix a semilinear monotone homogeneous operator $\shapley$ and take the set $\slin = \{x \in \R^{n} \colon x \le \shapley(x) \}$. This set is semilinear. Moreover, $\slin$ is closed because $\shapley$ is continuous. To prove that this is an real tropical cone, fix a pair $\lambda, \mu \in \R$ and $x, y \in \slin$. Since $\shapley$ is monotone and homogeneous, we have $\shapley(\max\{\lambda + x, \mu + y \}) \ge \shapley(\lambda + x) = \lambda + \shapley(x) \ge \lambda + x$ and similarly $\shapley(\max\{\lambda + x , \mu + y\}) \ge \mu + y$. Hence $\max\{\lambda + x, \mu + y \} \in \slin$.
\end{proof}

\begin{remark}
One could ask if there is a more direct way to obtain a piecewise description of the operator $\shapley$ given a real tropical cone $\salg$ (without the use of model theory). This can be done in the following way. We first decompose $\salg = \bigcup_{s = 1}^{p} \{x \in \R^{n} \colon A^{(s)}x \le b^{(s)}\}$ where the matrix $A^{(s)}$ has rational entries, and $b^{(s)}$ is a real vector. Then, given $x \in \R^{n}$ we denote by $P(x) \subset [p]$ the set of all $s \in [p]$ such that the polyhedron $\{ y \colon A^{(s)}y \le b^{(s)}, y \le x \}$ is nonempty. By the strong duality of linear programming (and the fact that $\shapley(x)$ is well defined for all $x \in \R^{n}$) we have
\begin{align*}
\shapley_{k}(x) &= \max_{s \in P(x)} \max\{y_{k} \colon A^{(s)}y \le b^{(s)}, y \le x \} \\
&= \max_{s \in P(x)} \min \{z^{\transpose}b^{(s)} + w^{\transpose}x \colon (A^{(s)})^{\transpose}z + w = \stbase_{k}, z \ge 0, w \ge 0 \} \, .
\end{align*}
For every $s \in [p]$, let $V^{(s)}_{k}$ denote the set of vertices of the rational polyhedron $\{ (z,w) \colon (A^{(s)})^{\transpose}z + w = \stbase_{k}, z \ge 0, w \ge 0\}$. Hence
\begin{equation}\label{eq:piecewise_description}
\shapley_{k}(x) = \max_{s \in P(x)} \min_{(z,w) \in V^{(s)}_{k}} \{z^{\transpose}b^{(s)} + w^{\transpose}x \} \, .
\end{equation} 
Moreover, by Farkas' Lemma, the polyhedron $\{ y \colon A^{(s)}y \le b^{(s)}, y \le x \}$ is nonempty if and only if for all $(z,w)$ such that $(A^{(s)})^{\transpose}z + w = 0$, $z \ge 0$, $w \ge 0$, we have $z^{\transpose}b^{(s)} + w^{\transpose}x \geq 0$. If $U^{(s)}$ consists of precisely one representative of every extreme ray of the rational cone $\{ (z,w) \colon (A^{(s)})^{\transpose}z + w = 0, z \ge 0, w \ge 0\}$, this amounts to the finite system of linear inequalities $z^{\transpose}b^{(s)} + w^{\transpose}x \geq 0$ for all $(z,w) \in U^{(s)}$. As a consequence, if we fix $I \subset [p]$, then the set $\polyh^{(I)}$ of all $x \in \R^{n}$ satisfying $P(x) = I$ is an intersection of half-spaces (both closed and open). By fixing the terms achieving the maximum and minimum in~\cref{eq:piecewise_description} we subdivide the sets $\polyh^{(I)}$ into smaller sets, $\polyh^{(I)} = \bigcup_{j = 1}^{N_{I}} \polyh^{(I)}_{j}$ such that every $\polyh^{(I)}_{j}$ is an intersection of half-spaces and $\shapley$ is affine on $\polyh^{(I)}_{j}$. Since $\shapley$ is continuous (by \cref{lemma:nonexpansive}), we can then restrict ourselves to these sets $\polyh^{(I)}_{j}$ that are full dimensional, and this gives the piecewise description of $\shapley$. 
\end{remark}

\begin{figure}[t]
\centering
\begin{tikzpicture}
\centering
\begin{scope}[scale = 0.6]
      \draw[gray!60, ultra thin] (-4.5,-4.5) grid (2.5,2.5);
       \fill[fill=lightgray, fill opacity = 1.0]
        (-3,0) -- (0,0) -- (0, -4) -- (1,-1) -- (1,1) -- cycle;
        \draw[very thick, line join=round] (-3,0) -- (0,0) -- (0, -4) -- (1,-1) -- (1,1) -- cycle;
\end{scope}
\end{tikzpicture}
\vspace*{-0.3cm}
\caption{A real tropical cone from \cref{ex:tropical_cone} (for $x_{3} = 0$).}\label{fig:tropical_cone}
\end{figure}
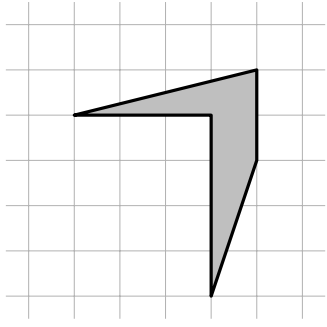

\begin{example}\label{ex:tropical_cone}
We illustrate our results on the following example. Take $n =3$, $p = 2$, $M_{k} = 1$, $S_{k,1} = \{1,2\}$ for all $k \in \{1,2,3\}$,
\begin{align*}
A^{(1)} &= 
\begin{bmatrix}
0 & 0 & 1 \\
1/4 & 0 & 3/4 \\
1 & 0 & 0
\end{bmatrix} \,
\quad b^{(1)} = 
\begin{bmatrix}
1 \\
3/4 \\
0
\end{bmatrix} \, , \\
A^{(2)} &= 
\begin{bmatrix}
0 & 1/3 & 2/3 \\
0 & 0 & 1 \\
0 & 1 & 0
\end{bmatrix} \,
\quad b^{(2)} = 
\begin{bmatrix}
4/3 \\
2\pi \\
0
\end{bmatrix} \, .
\end{align*}
Then, the operator $\shapley \colon \R^{3} \to \R^{3}$ is given by
\begin{align*}
\shapley_{1}(x) &= \max\Bigl\{x_{3} + 1, \frac{1}{3}x_{2} + \frac{2}{3}x_{3} + \frac{4}{3}\Bigr\} \, , \\
\shapley_{2}(x) &= \max\Bigl\{\frac{1}{4}x_{1} + \frac{3}{4}x_{3} + \frac{3}{4}, x_{3} + 2\pi \Bigr\} \, ,\\
\shapley_{3}(x) &= \max\{x_{1}, x_{2}\} \, .
\end{align*}
The real tropical cone $\{x \in \R^{3} \colon x \le \shapley(x) \}$ is depicted in \cref{fig:tropical_cone}.
\end{example}

\subsection{Description of real tropical cones by directed graphs}

We now describe how semilinear monotone homogeneous operators can be encoded by directed graphs. To this end we take a directed graph $\dgraph \coloneqq (\vertices, \edges)$, where the set of vertices is divided into Max vertices, Min vertices, and Random vertices, i.e., 
$\vertices \coloneqq \Minvertices \dunion\Randvertices \dunion \Maxvertices$,
where the symbol $\dunion$ denotes the disjoint union of sets.
We suppose that the sets of Max vertices and Min vertices are nonempty. If $\vertex \in \vertices$ is a vertex of $\dgraph$, then by 
$\inedge(\vertex):= \{(\vertexII,\vertex)\colon (\vertexII,\vertex)\in \edges \}$
we denote the set of its incoming edges, 
and by $\outedge(\vertex):= \{(\vertex,\vertexII)\colon (\vertex,\vertexII)\in \edges \}$ we denote the set of its outgoing edges.
We suppose that the every vertex has at least one outgoing edge. If $\vertex$ is a Min vertex or a Max vertex and $\edge \in \outedge(\vertex)$ is its outgoing edge, then we equip this edge with a real number $\payoff_{\edge}$. Furthermore, if $\vertex$ is a Random vertex, then we equip its set of outgoing edges with a rational probability distribution. More precisely, every edge $\edge \in \outedge(\vertex)$ is equipped with a strictly positive rational number $\probII_{\edge} \in \Q$, $\probII_{\edge} > 0$, and we suppose that $\sum_{\edge \in \outedge(\vertex)} \probII_{\edge} = 1$. 
We also make the following assumptions:
\begin{assumption}\label{assumption:graph}
\begin{assumpenum}
\item\label{item:graph1} Every path between any two Min vertices contains at least one Max vertex; 
\item\label{item:graph2} Every path between any two Max vertices contains at least one Min vertex;
\item\label{item:graph3} From every Random vertex, there is a path to a Min or a Max vertex.
\end{assumpenum}
\end{assumption}

We now  construct a semilinear monotone homogeneous operator from such a graph. We define a Markov chain with state space $\vertices$, and transition probabilities $\prob_{\vertex\vertex} \coloneqq 1$ for all $\vertex \in \Maxvertices \dunion \Minvertices$, $\prob_{\vertex\vertexII} \coloneqq \probII_{(\vertex,\vertexII)}$ if $\vertex \in \Randvertices$ and $(\vertex,\vertexII) \in \outedge(\vertex)$, and $\prob_{\vertex\vertexII} \coloneqq 0$ otherwise. Therefore, every
state of $\Maxvertices \dunion \Minvertices$ is absorbing,
and a trajectory of the Markov chain visits the 
states of $\Randvertices$ by picking at random,
for each vertex $\vertex\in \Randvertices$, 
one edge in $\outedge(\vertex)$
according to the probability law given by $\probII_{(\vertex,\cdot)}$,
until it reaches a state of $\Maxvertices \dunion \Minvertices$.
In this way, after leaving a Min vertex, the trajectory reaches a Max vertex, and vice versa.
If $\edge$ is an edge and $\vertex$ is a Max or Min vertex, then we denote by $\prob^{\edge}_{\vertex}$ the conditional probability to reach the absorbing state $\vertex$ from the head of $\edge$. Note that every $\prob^{\edge}_{\vertex}$ is rational since we have assumed that the $\probII_{\edge}$ are in $\Q$~\cite[Theorem~3.3.7]{kemeny_snell}. For the sake of simplicity, we assume that $\Minvertices = [n]$ and $\Maxvertices = [m]$. 

\begin{definition}\label{def:shapley_from_graphs}
The \emph{operator encoded by $\dgraph$} is the function $\shapley \colon \R^{n} \to \R^{n}$ defined as
\begin{equation}
\forall \vertex \in [n], \, (\shapley(x))_{\vertex} \coloneqq \min_{\edge \in \outedge(\vertex)}\Bigl( \payoff_{\edge} + \sum_{\vertexII \in [m]}\prob^{\edge}_{\vertexII}\max_{\edgeII \in \outedge(\vertexII)}\bigl( \payoff_{\edgeII} + \sum_{\vertexIII \in [n]} \prob^{\edgeII}_{\vertexIII}x_{\vertexIII} \bigr) \Bigr) \, . \label{eq:shapley_operator}
\end{equation}
\end{definition}

\begin{lemma}\label{lemma:shapley_from_graphs}
The operator encoded by $\dgraph$ is semilinear monotone homogeneous.
\end{lemma}

\begin{proof}
Let $\shapley$ be the operator encoded by $\dgraph$. It is obviously semilinear (as a definable function in $\logroups$) and monotone. We already observed that the Max and Min vertices are absorbing states in the Markov chain constructed from $\dgraph$. Besides, \cref{item:graph3} still holds in the subgraph obtained by removing the edges going out of the Max and Min vertices. As a consequence, for every Random vertex $\vertex$, the probability to reach a Min or Max vertex starting from $\vertex$ is positive. We deduce that the Max and Min vertices are the only final classes in the Markov chain. Let $\vertex$ be a Min vertex, and $\edge \in \outedge(\vertex)$. We claim that if $\vertexIII \in [n]$ is a Min state,
then $\prob^{\edge}_\vertexIII = 0$. Indeed, by \cref{item:graph1}, any path from the head of $\edge$ to $\vertexIII$ in $\dgraph$ contains a Max vertex. As a consequence, there is no path from the head of $\edge$ to $\vertexIII$ in the subgraph in which we have removed the edges going out of the Max and Min vertices. We deduce that for every Min vertex $\vertex$ and edge $\edge \in \outedge(\vertex)$, we have $\sum_{\vertexII \in [m]}\prob^{\edge}_{\vertexII} = 1$. Analogously, we can show that for all Max vertices $\vertexII$ and edge $\edgeII \in \outedge(\vertexII)$, $\sum_{\vertexIII \in [n]}\prob^{\edgeII}_{\vertexIII} = 1$. We deduce that the operator $\shapley$ is homogeneous.
\end{proof}

In the following lemma, we show that any semilinear monotone homogeneous operator is encoded by some digraph:
\begin{lemma}
Let $\shapley \colon \R^{n} \to \R^{n}$ be a semilinear monotone homogeneous operator. Then, there exists a directed graph $\dgraph$ satisfying \cref{assumption:graph} such that $\shapley$ is encoded by $\dgraph$.
\end{lemma}

\begin{proof}
The idea is to identify the representation~\cref{eq:semilineargames} to a special case of~\cref{eq:shapley_operator}, in which the probabilities $\prob^{\edge}_{\vertexII} $ with $\edge \in \outedge(\vertex)$ and $\vertex\in\Minvertices$ take only the values $0$ and $1$. 
Formally, 
let $A^{(1)}, \dots, A^{(p)} \in \Q^{n \times n}$ and $b^{(1)}, \dots, b^{(p)} \in \R^n$ such that \cref{lemma:semilinear_min_max} holds. We build $\dgraph$ as the graph in which the set of Min vertices is $[n]$, the set of Max vertices is $\dunion_k [M_k]$, and the set of Random vertices is $\dunion_{k \in [n], \, i \in [M_k]} S_{ki}$.
Let $k$ be a Min vertex. We add an edge $(k,i)$ for every $i \in [M_k]$, with $\payoff_{(k,i)} \coloneqq 0$. Moreover, for every $i \in [M_k]$, we add an edge $(i,s)$ for each $s \in S_{ki}$, with $\payoff_{(i,s)} \coloneqq b^{(s)}_k$. Finally, if $i \in [M_k]$ and $s \in S_{ki}$, we add an edge $(s,l)$ with $\probII_{(s,l)} \coloneqq A^{(s)}_{kl}$ for every $l \in [n]$ such that $A^{(s)}_{kl} > 0$. The requirements of \cref{assumption:graph} are straightforwardly satisfied. 
\end{proof}

\begin{figure}[t]
\centering
\begin{tikzpicture}[scale=0.65,>=stealth',max/.style={draw,rectangle,minimum size=0.5cm},min/.style={draw,circle,minimum size=0.5cm},av/.style={draw,diamond,minimum size=0.5cm}] 

\node[min] (min3) at (0.8, 4.5) {$3$};
\node[min] (min1) at (1.4, -2.5) {$1$};
\node[min] (min2) at (1.5, 0.5) {$2$};

\node[max] (max1) at (-1.5, -1) {$1$};
\node[max] (max2) at (4, 2) {$2$};
\node[max] (max3) at (4, -1) {$3$};

\node[av] (av13) at (5.5, -2){};
\node[av] (av23) at (-1, 2){};

\draw[->] (min1) to node[below, font=\small]{} (max1);
\draw[->] (min2) to node[above left = -3ex, font=\small]{} (max2);
\draw[->] (min3) to node[below, font=\small]{} (max3);

\draw[->] (max3) to node[above right=-1ex, font=\small]{} (min1);
\draw[->] (max3) to node[above right=-1ex, font=\small]{} (min2);

\draw[->] (max2) to node[above right=-1ex, font=\small]{$3/4$} (av13);

\draw[->] (max2) to node[above right=-1ex, font=\small]{$2\pi$} (min3);

\draw[->] (av13) to node[below, font=\small]{$1/4$} (min1);
\draw[->] (av13) to[out = 10, in = 50] node[above right, font=\small]{$3/4$} (min3);

\draw[->] (max1) to[out = 140, in = 170] node[above=0.5ex, font=\small]{$1$} (min3);
\draw[->] (max1) to node[below right=-0.8ex, font=\small]{$4/3$} (av23);

\draw[->] (av23) to node[above, font=\small]{$1/3$} (min2);
\draw[->] (av23) to node[above left=-1ex, font=\small]{$2/3$} (min3);

\end{tikzpicture}
\caption{Graph that encodes the operator from \cref{ex:tropical_cone}. Min vertices are depicted by circles, Max vertices are depicted by squares, Random vertices are depicted by diamonds. We put $\payoff_{\edge} = 0$ for every edge $\edge \in \edges$ that has no label.}\label{fig:graph}
\end{figure}
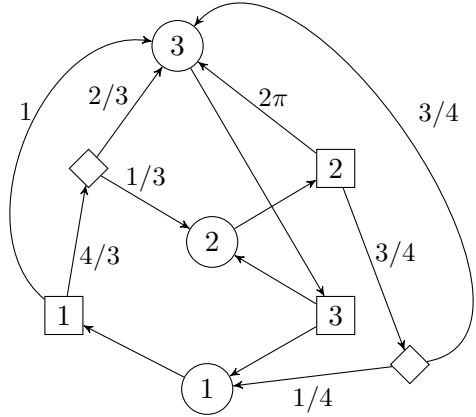

\begin{example}
The graph presented in \cref{fig:graph} encodes the operator from \cref{ex:tropical_cone}.
\end{example}

\subsection{Construction of tropical Metzler spectrahedra}
The following proposition characterizes the semilinear monotone homogeneous operators associated with tropical Metzler spectrahedral cones. We discussed this family of operators in our previous work~\cite{issac2016jsc}, where we interpreted them as the dynamic programming operators of a zero-sum game.

\begin{proposition}\label{proposition:tropical_metzler_graphs}
Suppose that the graph $\dgraph$ fulfills \cref{assumption:graph} and has the following properties:
\begin{itemize}
\item every Random vertex has exactly two outgoing edges and the probability distribution associated with these edges is equal to $(1/2,1/2)$
\item every edge outgoing from a Random vertex has a Max vertex as its head.
\end{itemize}
Let $\shapley$ denote the semilinear monotone homogeneous operator encoded by $\dgraph$. Then, the set $\{x \in \trop^{n} \colon x \le \shapley(x)\}$ is a tropical Metzler spectrahedral cone. 

\end{proposition}

\begin{proof}
Consider the Markov chain introduced before \cref{def:shapley_from_graphs}, take a Min vertex $v \in [n]$ and an outgoing edge $e \in \outedge(v)$. Under the assumptions over the graph $\dgraph$, the absorbing states reachable from the head of $e$ form a set $\{w_e, w'_e\} \subset [m]$ of cardinality at most $2$ (we use the convention $w_e = w'_e$ if there is only one such absorbing state). Moreover, if $w_e \neq w'_e$, then $p_{w_e}^e = p_{w'_e}^e = 1/2$. Furthermore, observe that if $w \in [m]$ is a Max vertex and $e' \in \outedge(w)$ is an outgoing edge, then our assumptions imply that the head of $e'$ is a Min vertex. We denote it by $u_{e'}$. With this notation, we have:
\begin{align}\label{eq:shapleymetzler}
(\shapley(x))_{\vertex} = \min_{\edge \in \outedge(\vertex)}\Bigl( \payoff_{\edge} + \frac{1}{2}\bigl(\max_{\edgeII \in \outedge(\vertexII_{\edge})}( \payoff_{\edgeII} + x_{\vertexIII_{\edgeII}} ) + \max_{\edgeII \in \outedge(\vertexII'_{\edge})}( \payoff_{\edgeII} + x_{\vertexIII_{\edgeII}} )\bigr) \Bigr) \, .
\end{align}
The operators of the form given in \cref{eq:shapleymetzler} are studied in \cite[Sections~4.2 and~5.1]{issac2016jsc}. In particular, the claim follows from \cite[Lemma~52]{issac2016jsc}. 
\end{proof}

We want to show that every real tropical cone associated with a graph $\dgraph$ is a projection of a tropical Metzler spectrahedron. The idea of the proof is to take an arbitrary graph $\dgraph$ and transform it (by adding auxiliary states) into a graph $\trdgraph$ that fulfills the conditions of \cref{proposition:tropical_metzler_graphs}. Furthermore, our construction needs to preserve the projection. A key ingredient
is the following construction, which was used by Zwick and Paterson~\cite{zwick_paterson} to show the reduction from discounted games to simple stochastic games.

\begin{figure}[t]
\begin{center}
\begin{minipage}{0.45\textwidth}
\centering
\begin{tikzpicture}[scale=0.9,>=stealth',row/.style={draw,diamond,minimum size=0.5cm},col/.style={draw,rectangle,minimum size=0.5cm},av/.style={draw, circle,fill, inner sep = 0pt,minimum size = 0.2cm}]

\node[row] (i1) at (5, 5) {$v$};

\node (k1) at (3, 7) {$g$};
\node (k2) at (3, 3) {$h$};

\node (k4) at (7, 5) {$b$};
\node (k5) at (7, 3) {$c$};
\node (k6) at (7, 7) {$a$};

\draw[->] (k1) to (i1);
\draw[->] (k2) to (i1);

\draw[->] (i1) to (k4);
\draw[->] (i1) to (k5);
\draw[->] (i1) to (k6);

\end{tikzpicture}
\end{minipage}\hfill
\begin{minipage}{0.54\textwidth}
\begin{center}
\begin{tikzpicture}[scale=0.9,>=stealth',row/.style={draw,diamond,minimum size=0.5cm},col/.style={draw,rectangle,minimum size=0.5cm},av/.style={draw, circle,fill, inner sep = 0pt,minimum size = 0.2cm}]

\node[row] (i1) at (4, 5) {$w$};

\node[row] (i2) at (5.5, 5) {$u$};

\node (k1) at (3, 7) {$g$};
\node (k2) at (3, 3) {$h$};

\node (k4) at (7, 5) {$b$};
\node (k5) at (7, 3) {$c$};
\node (k6) at (7, 7) {$a$};

\draw[->] (k1) to (i1);
\draw[->] (k2) to (i1);

\draw[->] (i1) to (k5);
\draw[->] (i1) to (i2);
\draw[->] (i2) to (k4);
\draw[->] (i2) to (k6);

\end{tikzpicture}
\end{center}
\end{minipage}
\end{center}
\caption{Lowering the degree. Random vertices are depicted by diamonds.}\label{fig:lowering_degree}
\end{figure}
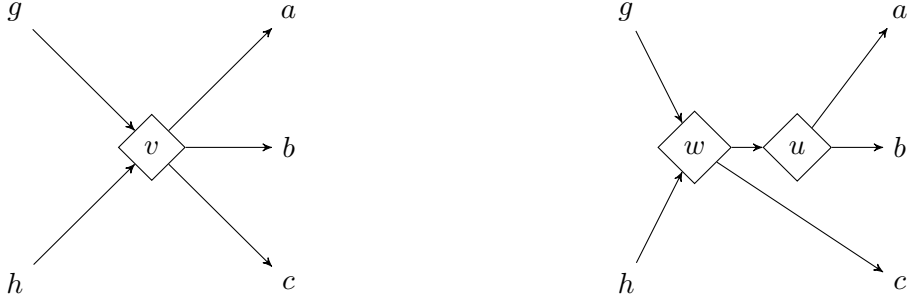

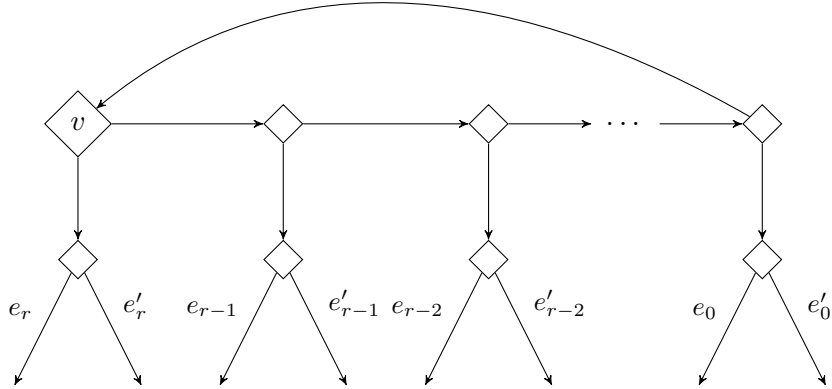
\begin{figure}[t]
\begin{center}
\centering
\begin{tikzpicture}[scale=0.9,>=stealth',row/.style={draw,diamond,minimum size=0.5cm},col/.style={draw,rectangle,minimum size=0.5cm},av/.style={draw, circle,fill, inner sep = 0pt,minimum size = 0.2cm}]

\node[row] (i1) at (1, 4) {$v$};
\node[row] (i2) at (4,4) {};
\node[row] (i3) at (7,4) {};
\node[row] (i4) at (11,4) {};

\node[row] (j1) at (1,2) {};
\node[row] (j2) at (4,2) {};
\node[row] (j3) at (7,2) {};
\node[row] (j4) at (11,2) {};

\node (k1) at (0, 0) {};
\node (k2) at (2, 0) {};
\node (k3) at (3, 0) {};
\node (k4) at (5, 0) {};
\node (k5) at (6, 0) {};
\node (k6) at (8, 0) {};
\node (k7) at (10, 0) {};
\node (k8) at (12, 0) {};

\coordinate (l1) at (8.5,4);
\coordinate (l2) at (9.5,4);

\draw[->] (i1) to (i2);
\draw[->] (i2) to (i3);

\draw[->] (i1) to (j1);
\draw[->] (i2) to (j2);
\draw[->] (i3) to (j3);
\draw[->] (i4) to (j4);

\draw[->] (j1) to node[above left, font=\small]{$\edge_{r}$} (k1);
\draw[->] (j1) to node[above right, font=\small]{$\edgeII_{r}$} (k2);
\draw[->] (j2) to node[above left, font=\small]{$\edge_{r-1}$} (k3);
\draw[->] (j2) to node[above right, font=\small]{$\edgeII_{r-1}$} (k4);
\draw[->] (j3) to node[above left, font=\small]{$\edge_{r-2}$} (k5);
\draw[->] (j3) to node[above right, font=\small]{$\edgeII_{r-2}$} (k6);
\draw[->] (j4) to node[above left, font=\small]{$\edge_{0}$} (k7);
\draw[->] (j4) to node[above right, font=\small]{$\edgeII_{0}$} (k8);

\draw[->] (i3) to (l1);
\draw[->] (l2) to (i4);

\draw[->] (i4) to[out = 150, in = 40] (i1);

\path (l1) -- node[auto=false]{\ldots} (l2);

\end{tikzpicture}
\end{center}
\caption{The construct of Zwick and Paterson.}\label{fig:zwick_paterson}
\end{figure}

\begin{lemma}[\cite{zwick_paterson}]\label{lemma:zwick_paterson}
One can transform an arbitrary graph $\dgraph$ into a graph $\trdgraph$ such that
\begin{itemize}
\item every Random vertex of $\trdgraph$ has exactly two outgoing edges and the probability distribution associated with these edges is equal to $(1/2,1/2)$
\item $\dgraph$ and $\trdgraph$ encode the same operator.
\end{itemize}
\end{lemma}

Let us present the construction of Zwick and Paterson for the sake of completeness.

\begin{proof}
Fix a Random vertex $\vertex$ belonging to $\dgraph$. If this vertex has only one outgoing edge $\edge$, then we can delete $\vertex$ by joining all incoming edges $\inedge(\vertex)$ with the head of $\edge$. 

If $\vertex$ has at least three outgoing edges, then we enumerate the outgoing edges $\outedge(\vertex)$ by $\{\edge_{1}, \dots, \edge_{d} \}$, $d \ge 3$. Let us recall that the vertex $\vertex$ is equipped with a probability distribution $(\probII_{\edge_{s}})_{s = 1}^{d}$. We now perform the transformation presented on \cref{fig:lowering_degree}. We replace the vertex $\vertex$ by a pair of vertices $(\vertexII, \vertexIII)$ such that all incoming edges of $\vertex$ are connected to $\vertexII$ and $\vertexII$ has two outgoing edges: one going to the head of $\edge_{1}$ with probability $\probII_{\edge_{1}}$ and the other going to $\vertexIII$ with probability $1 - \probII_{\edge_{1}}$. Finally, $\vertexIII$ has $d - 1$ outgoing edges, the head of the $s$th outgoing edge is the head of $\edge_{s}$, and the associated probability is equal to $\probII_{\edge_{s}}/(1- \probII_{\edge_{1}})$. We repeat this transformation until we reach a graph in which all Random vertices have exactly two outgoing edges. 

If $\vertex$ has exactly two outgoing edges, then we denote the heads of these edges by $\vertexII$ and $\vertexIII$, and the associated probability distribution by $(\probII, 1 - \probII)$, where $\probII = a/b$, $a,b \in \N^{*}$ and $a < b$. If $\probII \neq 1/2$, then we take $r \ge 1$ such that $2^{r} \le b < 2^{r+1}$. We write $a$ and $b - a$ in binary, $a = \sum_{s = 0}^{r} c_{s} 2^{s}$ and $b - a = \sum_{s = 0}^{r} d_{s}2^{s}$ for $c_{s}, d_{s} \in \{0,1\}$. We now replace the outgoing edges of vertex $\vertex$ by the construct presented on \cref{fig:zwick_paterson}. In this construction, every Random node has exactly two outgoing edges and the associated probability distribution is equal to $(1/2,1/2)$. Furthermore, for any $s$, if $c_{s} = 1$, then head of $\edge_{s}$ is $\vertexII$ and if $c_{s} = 0$, then the head of $\edge_{s}$ is $\vertex$. Similarly, if $d_{s} = 1$, then head of $\edgeII_{s}$ is $\vertexIII$ and if $d_{s} = 0$, then the head of $\edgeII_{s}$ is $\vertex$. Suppose that the Markov chain reaches $\vertex$. Then, with probability $a/2^{r+2}$ the Markov chain goes to $\vertexII$ without coming back to $\vertex$. Similarly, with probability $(b-a)/2^{r+2}$ the Markov chain moves to $\vertexIII$ without coming back to $\vertex$. Therefore, the probability that the Markov chain finally reaches $\vertexII$ is equal to $a/b$ and the probability that it finally reaches $\vertexIII$ is equal to $(b - a)/b$. We repeat this procedure for every Random vertex of our graph.

To finish the proof, observe that the operations described above do not affect the associated semilinear monotone homogeneous operator.
\end{proof}

We now describe how to transform a graph given in \cref{lemma:zwick_paterson} into a graph that verifies the conditions of \cref{proposition:tropical_metzler_graphs}. More precisely, we transform the graph $\dgraph$ (which has $n$ Min vertices) into a graph $\trdgraph$ (which has $n'$ Min vertices, where $n' \geq n$) in such a way that the real tropical cone $\{x \in \R^{n} \colon x \le \shapley(x)\}$ associated with $\dgraph$ is a projection of the real tropical cone $\{x \in \R^{n'} \colon x \le \trshapley(x)\}$ associated with $\trdgraph$. 
The main difficulty here is that the operators arising from tropical (Metzler)
spectrahedra have a special structure, of ``Player I -- Chance -- Player II''
type,  to adopt a game theoretical terminology, meaning that arcs in the graph
connect Min vertices to Random vertices, Random vertices
to Max vertices, and Max vertices to Min vertices,
as is apparent from~\cref{eq:shapleymetzler}. By comparison, 
the Zwick--Paterson
construction (\cref{lemma:zwick_paterson}) leads to a graph with 
consecutive sequences of Random nodes. We shall see, however, that
the latter situation can be reduced from the former one
by applying, as a basic ingredient, two transformations, the validity of which
is expressed in \cref{lemma:first_transformation,lemma:second_transformation}.

The first transformation that we execute is presented on \cref{fig:changing_owners}. It is given as follows. Suppose that we are given a graph $\dgraph$. Denote $\Minvertices = [n]$ and $\Maxvertices = [m]$. Furthermore, let $\edges_{\Max} \subset \edges$ denote the set of all edges that have a Max vertex as their tail. Let $\shapley$ denote the operator associated with $\dgraph$. For every Max vertex $\vertex \in [m]$ and outgoing edge $\edge \in \outedge(\vertex)$, we insert a Min vertex between $\vertex$ and the head of $\edge$, as illustrated in \cref{fig:changing_owners}. In a similar way, for every Min vertex $\vertex$ and incoming edge $\edge \in \inedge(\vertex)$, we insert a Max vertex between the tail of $\edge$ and $\vertex$. We denote the transformed graph by $\trdgraph$. Observe that this graph fulfills \cref{assumption:graph}. 
We refer to the Min vertices in $\trdgraph$ as follows: the vertices that were present in $\dgraph$ are denoted by $[n]$, whereas the added Min vertices are denoted by $\edge \in \edges_{\Max}$.

\begin{figure}[t]
\begin{center}
\begin{minipage}{0.45\textwidth}
\centering
\begin{tikzpicture}[scale=0.9,>=stealth',row/.style={draw,circle,minimum size=0.5cm},col/.style={draw,rectangle,minimum size=0.5cm},av/.style={draw, circle,fill, inner sep = 0pt,minimum size = 0.2cm}]

\node[row] (i1) at (1.25, 2.5) {$1$};
\node[row] (i2) at (1.25, 1.25) {$2$};
\node[row] (i3) at (1.25, 0) {$3$};

\node[col] (j1) at (6.25, 2.5) {$1$};
\node[col] (j2) at (6.25, 1.25) {$2$};
\node[col] (j3) at (6.25, 0) {$3$};

\node (k1) at (2.25, 2.5) {$a$};
\node (k2) at (2.25, 1.25) {$b$};
\node (k3) at (2.25, 0) {$c$};

\node (k4) at (5.25, 2.5) {$f$};
\node (k5) at (5.25, 1.25) {$g$};
\node (k6) at (5.25, 0) {$h$};

\draw[->] (k1) to (i1);
\draw[->] (k2) to (i2);
\draw[->] (k3) to (i3);

\draw[->] (j1) to (k4);
\draw[->] (j2) to (k5);
\draw[->] (j3) to (k6);

\end{tikzpicture}
\end{minipage}\hfill
\begin{minipage}{0.54\textwidth}
\begin{center}
\begin{tikzpicture}[scale=0.9,>=stealth',row/.style={draw,circle,minimum size=0.5cm},col/.style={draw,rectangle,minimum size=0.5cm},av/.style={draw, circle,fill, inner sep = 0pt,minimum size = 0.2cm}]

\node[row] (i1) at (1, 2.5) {$1$};
\node[row] (i2) at (1, 1.25) {$2$};
\node[row] (i3) at (1, 0) {$3$};

\node[row] (i4) at (7, 2.5) {$1'$};
\node[row] (i5) at (7, 1.25) {$2'$};
\node[row] (i6) at (7, 0) {$3'$};

\node[col] (j1) at (8.5, 2.5) {$1$};
\node[col] (j2) at (8.5, 1.25) {$2$};
\node[col] (j3) at (8.5, 0) {$3$};

\node[col] (j4) at (2.5, 2.5) {$1'$};
\node[col] (j5) at (2.5, 1.25) {$2'$};
\node[col] (j6) at (2.5, 0) {$3'$};

\node (k1) at (4, 2.5) {$a$};
\node (k2) at (4, 1.25) {$b$};
\node (k3) at (4, 0) {$c$};

\node (k4) at (5.5, 2.5) {$f$};
\node (k5) at (5.5, 1.25) {$g$};
\node (k6) at (5.5, 0) {$h$};

\draw[->] (k1) to (j4);
\draw[->] (k2) to (j5);
\draw[->] (k3) to (j6);

\draw[->] (i4) to (k4);
\draw[->] (i5) to (k5);
\draw[->] (i6) to (k6);

\draw[->] (j4) to (i1);
\draw[->] (j5) to (i2);
\draw[->] (j6) to (i3);

\draw[->] (j1) to (i4);
\draw[->] (j2) to (i5);
\draw[->] (j3) to (i6);

\end{tikzpicture}
\end{center}
\end{minipage}
\end{center}
\caption{First transformation of a graph. Min vertices are presented by circles, Max vertices are presented by squares.}\label{fig:changing_owners}
\end{figure}
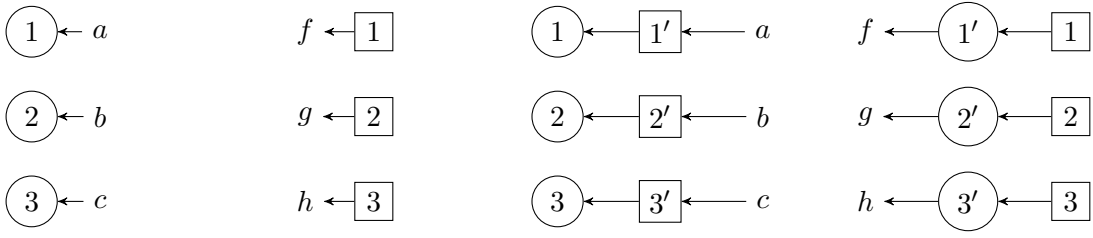

\begin{lemma}\label{lemma:first_transformation}
Suppose that the operator $\trshapley$ is obtained from $\shapley$ by the first transformation above. 
Then, the real tropical cone $\{x \in \R^{n} \colon x \le \shapley(x)\}$ is the projection of the real tropical cone $\{(x,x') \in \R^{n} \times \R^{\card{\edges_{\Max}}} \colon (x,x') \le \trshapley(x,x')\}$.
\end{lemma}
\begin{proof}
Denote the operator $\shapley$ as
\[
\forall \vertex \in [n], \ (\shapley(x))_{\vertex} = \min_{\edge \in \outedge(\vertex)}\Bigl( \payoff_{\edge} + \sum_{\vertexII \in [m]}\prob^{\edge}_{\vertexII}\max_{\edgeII \in \outedge(\vertexII)}\bigl( \payoff_{\edgeII} + \sum_{\vertexIII \in [n]} \prob^{\edgeII}_{\vertexIII}x_{\vertexIII} \bigr) \Bigr) \, .
\]
Observe that for every $\vertex \in [n]$ we have
\[
(\trshapley(x,x'))_{\vertex} = \min_{\edge \in \outedge(\vertex)}\Bigl( \payoff_{\edge} + \sum_{\vertexII \in [m]} \prob^{\edge}_{\vertexII}\max_{\edgeII \in \outedge(\vertexII)}\bigl( \payoff_{\edgeII} + x_{\edgeII} \bigr) \Bigr) \, .
\]
Furthermore, for every $\edge \in \edges_{\Max}$ we have
\[
(\trshapley(x,x'))_{\edge} = \sum_{\vertex \in [n]} \prob^{\edge}_{\vertex}x_{\vertex} \, .
\]
Therefore, if $x \le \shapley(x)$ and for every $\edge \in \edges_{\Max}$ we set $x_{\edge} = \sum_{\vertex \in [n]} \prob^{\edge}_{\vertex}x_{\vertex}$, then for every $\vertex \in [n]$ we have $x_{\vertex} \le (\shapley(x))_{\vertex} = (\trshapley(x,x'))_{\vertex}$ and for every $\edge \in \edges_{\Max}$ we have $x_{\edge} = (\trshapley(x,x'))_{\edge}$. Conversely, if $(x,x') \le \trshapley(x,x')$, then we have $x_{\vertex} \le (\trshapley(x,x'))_{\vertex} \le (\shapley(x))_{\vertex}$ for every $\vertex \in [n]$.
\end{proof}

The second transformation is given as follows. As previously, suppose that we are given a graph $\dgraph$. Denote $\Minvertices = [n]$ and $\Maxvertices = [m]$. Furthermore, let $\edges_{\Max} \subset \edges$ denote the set of all edges that have a Max vertex as their tail. Let $\shapley$ denote the operator associated with $\dgraph$. Moreover, suppose that $\dgraph$ is such that every edge $\edge \in \edges_{\Max}$ has a Min vertex as its head. Suppose that $\edge^{*} \in \edges$ is a fixed edge in $\dgraph$ that connects two Random vertices. We add a Max vertex $m+1$ and a Min vertex $n + 1$ onto $\edge^{*}$ as presented on \cref{fig:adding_states}. We denote the transformed graph by $\trdgraph$. Since every edge $\edge \in \edges_{\Max}$ has a Min vertex as its head, every path that joins a Max vertex with a Min vertex has length $1$. In particular, $\edge^{*}$ does not belong to any such path. Hence, the transformed graph $\trdgraph$ fulfills \cref{assumption:graph}. 

\begin{figure}[t]
\begin{center}
\begin{minipage}{0.45\textwidth}
\centering
\begin{tikzpicture}[scale=0.9,>=stealth',row/.style={draw,circle,minimum size=0.5cm},col/.style={draw,rectangle,minimum size=0.5cm},av/.style={draw, circle,fill, inner sep = 0pt,minimum size = 0.2cm}]

\node (j1) at (3.25, 1.25) {$a$};
\node (j2) at (5.25, 1.25) {$b$};

\draw[->] (j1) to node[above = 0ex, font = \small] {$e^{*}$} (j2);

\end{tikzpicture}
\end{minipage}\hfill
\begin{minipage}{0.54\textwidth}
\begin{center}
\begin{tikzpicture}[scale=0.9,>=stealth',row/.style={draw,circle,minimum size=0.5cm},col/.style={draw,rectangle,minimum size=0.5cm},av/.style={draw, circle,fill, inner sep = 0pt,minimum size = 0.2cm}]

\node (i1) at (3.25,1.25) {$a$};
\node[col]  (k1) at (5.5, 1.25) {$m+1$};
\node[row] (i2) at (7.5, 1.25) {$n+1$};

\node (j2) at (9.5, 1.25) {$b$};

\draw[->] (i1) to (k1);
\draw[->] (k1) to node {} (i2);

\draw[->] (i2) to node {} (j2);
\end{tikzpicture}
\end{center}
\end{minipage}
\end{center}
\caption{Second transformation of a graph.}\label{fig:adding_states}
\end{figure}
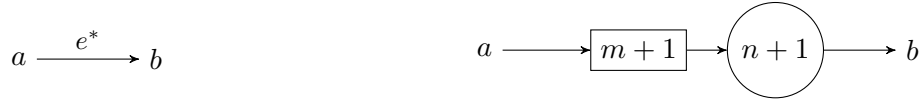

\begin{lemma}\label{lemma:second_transformation}
Suppose that the operator $\trshapley$ is obtained from $\shapley$ by the second transformation above. 
Then, the real tropical cone $\{x \in \R^{n} \colon x \le \shapley(x)\}$ is a projection of the real tropical cone $\{(x,x_{n+1}) \in \R^{n+1} \colon (x,x_{n+1}) \le \trshapley(x,x_{n+1})\}$.
\end{lemma}
\begin{proof}
Denote the operator $\shapley$ as
\[
\forall \vertex \in [n], \ (\shapley(x))_{\vertex} = \min_{\edge \in \outedge(\vertex)}\Bigl( \payoff_{\edge} + \sum_{\vertexII \in [m]}\prob^{\edge}_{\vertexII}\max_{\edgeII \in \outedge(\vertexII)}\bigl( \payoff_{\edgeII} + \sum_{\vertexIII \in [n]} \prob^{\edgeII}_{\vertexIII}x_{\vertexIII} \bigr) \Bigr) \, .
\]
Let us introduce the following notation. For every $\edge \in \edges$, we denote by $\prob^{\edge}_{\edge{*}}$ the conditional probability that the Markov chain reaches the head of $\edge^{*}$ from the head of $\edge$. Moreover, for every Max vertex $\vertexII$ and every edge $\edge \in \edges$, we denote by $\prob^{\edge}_{\vertexII2}$ the conditional probability that the Markov chain reaches $\vertexII$ from the head of $\edge$ without passing by the head of $\edge^{*}$. Thus, for every Min vertex $\vertexII$ and every $\edge \in \edges$ we have $\prob^{\edge}_{\vertexII} = \prob^{\edge}_{\edge{*}}\prob^{\edge^{*}}_{\vertexII} + \prob^{\edge}_{\vertexII 2}$. 
Therefore, for any $x_{n+1} \in \R$ we have
\[
(\trshapley(x, x_{n+1}))_{n+1} = \sum_{\vertexII \in [m]}\prob^{\edge^{*}}_{\vertexII}\max_{\edgeII \in \outedge(\vertexII)} \bigl( \payoff_{\edgeII} + \sum_{\vertexIII \in [n]} \prob^{\edgeII}_{\vertexIII}x_{\vertexIII} \bigr) = (\trshapley(x, 0))_{n+1}
\]
and
\begin{equation*}
\begin{aligned}
(\shapley(x))_{\vertex} &= \min_{\edge \in \outedge(\vertex)}\Bigl( \payoff_{\edge} + \sum_{\vertexII \in [m]}(\prob^{\edge}_{\edge^{*}}\prob^{\edge^{*}}_{\vertexII} + \prob^{\edge}_{\vertexII 2})\max_{\edgeII \in \outedge(\vertexII)}\bigl( \payoff_{\edgeII} + \sum_{\vertexIII \in [n]} \prob^{\edgeII}_{\vertexIII}x_{\vertexIII} \bigr) \Bigr) \\
&= \min_{\edge \in \outedge(\vertex)}\Bigl( \payoff_{\edge} + \prob^{\edge}_{\edge^{*}}(\trshapley(x, 0))_{n+1} + \sum_{\vertexII \in [m]}\prob^{\edge}_{\vertexII 2}\max_{\edgeII \in \outedge(\vertexII)}\bigl( \payoff_{\edgeII} + \sum_{\vertexIII \in [n]} \prob^{\edgeII}_{\vertexIII}x_{\vertexIII} \bigr) \Bigr) \, .
\end{aligned}
\end{equation*}

Furthermore, for every $\vertex \in [n]$ we have
\[
(\trshapley(x, x_{n+1}))_{\vertex} = \min_{\edge \in \outedge(\vertex)}\Bigl( \payoff_{\edge} + \prob^{\edge}_{\edge^{*}}x_{n+1} + \sum_{\vertexII \in [m]}\prob^{\edge}_{\vertexII 2}\max_{\edgeII \in \outedge(\vertexII)}\bigl( \payoff_{\edgeII} + \sum_{\vertexIII \in [n]} \prob^{\edgeII}_{\vertexIII}x_{\vertexIII} \bigr) \Bigr) \, .
\]
Therefore, if $x \le \shapley(x)$ and we set $x_{n+1} = (\trshapley(x,0))_{n+1}$, then $(x,x_{n+1}) \le \trshapley(x,x_{n+1})$. Conversely, if $(x,x_{n+1}) \le \trshapley(x,x_{n+1})$, then $x_{n+1} \le (\trshapley(x,0))_{n+1}$ and hence $x_{\vertex} \le (\shapley(x))_{\vertex}$ for all $\vertex \in [n]$.
\end{proof}

\begin{proposition}\label{lemma:main_stratum_cone}
Every closed semilinear real tropical cone is a projection of a real tropical Metzler spectrahedron.
\end{proposition}
\begin{proof}
Take any closed semilinear real tropical cone $\slin = \{x \in \R^{n} \colon x \le \shapley(x) \}$. Let $\dgraph$ denote the graph associated with $\shapley$. By \cref{lemma:zwick_paterson} we may suppose that the probabilities associated with Random vertices in $\dgraph$ are equal to $1/2$. We perform the first transformation on the graph $\dgraph$. Denote the transformed graph by $\dgraph_{1}$. We perform the second transformation on every edge in $\dgraph_{1}$ that joins two Random vertices. Denote the transformed graph by $\trdgraph$ and the associated operator as $\trshapley$. By \cref{lemma:first_transformation,lemma:second_transformation}, the real tropical cone $\{x \in \R^{n} \colon x \le \shapley(x) \}$ is the projection of the real tropical cone $\{(x,x') \in \R^{n} \times \R^{n'} \colon (x,x') \le \trshapley(x,x') \}$. Furthermore, $\trdgraph$ fulfills the conditions of \cref{proposition:tropical_metzler_graphs}. Therefore, the set $\slin' = \{(x,x') \in \trop^{n} \times \trop^{n'} \colon (x,x') \le \trshapley(x,x') \}$ is a tropical Metzler spectrahedral cone. Finally, we take the set 
\begin{equation*}
\begin{aligned}
\slin'' &= \{(x,x', y) \in \trop^{n} \times \trop^{n'} \times \trop^{n + n'} \colon (x,x') \le \trshapley(x,x') \land (x,x') + y \ge 0 \} \\
&= \{(x,x', y) \in \R^{n} \times \R^{n'} \times \R^{n + n'} \colon (x,x') \le \trshapley(x,x') \land (x,x') + y \ge 0 \} \, .
\end{aligned}
\end{equation*}
The set $\slin''$ is a real tropical Metzler spectrahedron. Moreover, $\slin$ is a projection of $\slin''$.
\end{proof}

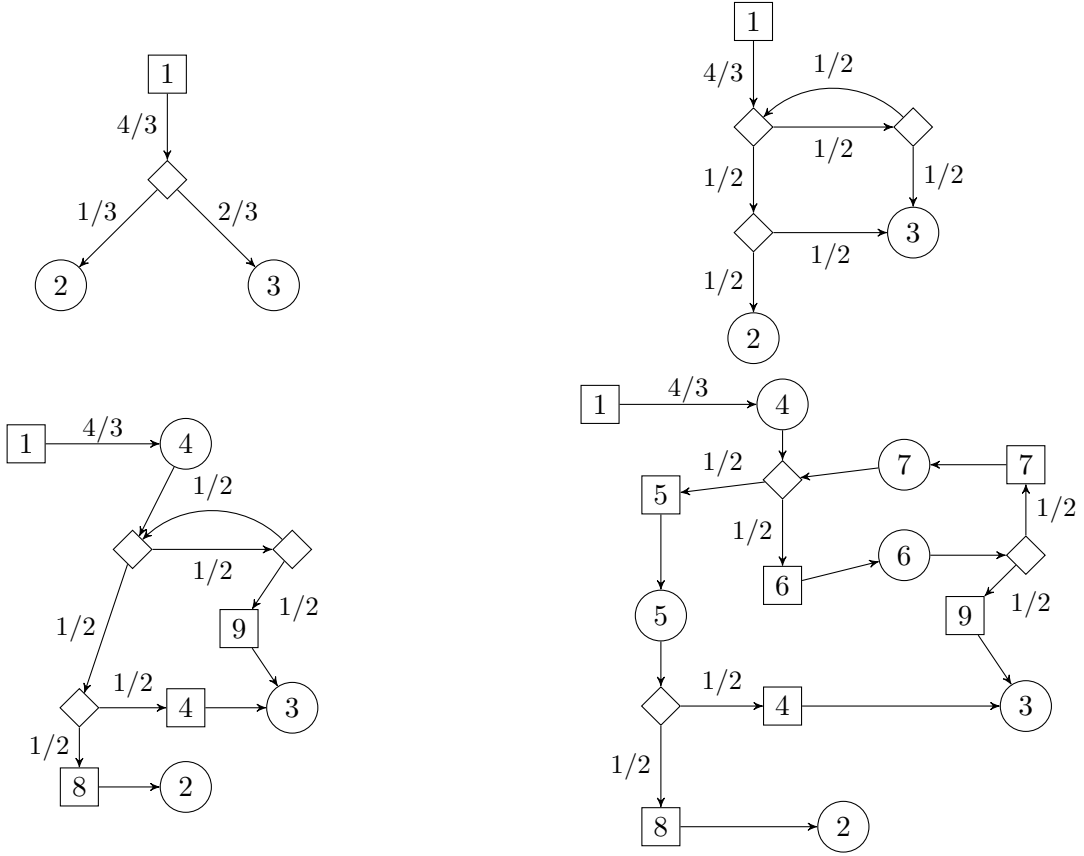
\begin{figure}[t]
\centering
\begin{minipage}{0.45\textwidth}
\centering
\begin{tikzpicture}[scale=0.7,>=stealth',max/.style={draw,rectangle,minimum size=0.5cm},min/.style={draw,circle,minimum size=0.5cm},av/.style={draw,diamond,minimum size=0.5cm}] 

\node[max] (max1) at (2, 4) {$1$};

\node[min] (min3) at (4, 0) {$3$};
\node[min] (min2) at (0, 0) {$2$};

\node[av] (av23) at (2, 2){};

\draw[->] (max1) to node[left, font=\small]{$4/3$} (av23);

\draw[->] (av23) to node[above left = -1ex, font=\small]{$1/3$} (min2);
\draw[->] (av23) to node[above right = -1ex, font=\small]{$2/3$} (min3);

\end{tikzpicture}
\end{minipage}\hfill
\begin{minipage}{0.45\textwidth}
\centering

\begin{tikzpicture}[scale=0.7,>=stealth',max/.style={draw,rectangle,minimum size=0.5cm},min/.style={draw,circle,minimum size=0.5cm},av/.style={draw,diamond,minimum size=0.5cm}] 

\node[max] (max1) at (2, 4) {$1$};

\node[min] (min3) at (5, 0) {$3$};
\node[min] (min2) at (2, -2) {$2$};

\node[av] (av23) at (2, 2){};
\node[av] (av4) at (2, 0){};
\node[av] (av5) at (5, 2){};

\draw[->] (max1) to node[left, font=\small]{$4/3$} (av23);

\draw[->] (av23) to node[left, font=\small]{$1/2$} (av4);
\draw[->] (av23) to node[below, font=\small]{$1/2$} (av5);
\draw[->] (av5) to node[right, font=\small]{$1/2$} (min3);

\draw[->] (av4) to node[left, font=\small]{$1/2$} (min2);
\draw[->] (av4) to node[below, font=\small]{$1/2$} (min3);

\draw[->] (av5) to[in = 45, out = 135] node[above, font=\small]{$1/2$} (av23);

\end{tikzpicture}
\end{minipage}\hfill

\begin{minipage}{0.45\textwidth}
\centering

\begin{tikzpicture}[scale=0.7,>=stealth',max/.style={draw,rectangle,minimum size=0.5cm},min/.style={draw,circle,minimum size=0.5cm},av/.style={draw,diamond,minimum size=0.5cm}] 

\node[max] (max1) at (0, 4) {$1$};

\node[min] (min4) at (3, 4) {$4$};
\node[max] (max2) at (1, -2.5) {$8$};
\node[max] (max3) at (4, 0.5) {$9$};

\node[max] (max4) at (3, -1) {$4$};

\node[min] (min3) at (5, -1) {$3$};
\node[min] (min2) at (3, -2.5) {$2$};

\node[av] (av23) at (2, 2){};
\node[av] (av4) at (1, -1){};
\node[av] (av5) at (5, 2){};

\draw[->] (max1) to node[above=-0.7ex, font=\small]{$4/3$} (min4);
\draw[->] (min4) to (av23);

\draw[->] (max2) to (min2);
\draw[->] (max3) to (min3);

\draw[->] (av5) to node[below right, font=\small]{$1/2$} (max3);

\draw[->] (av4) to node[left, font=\small]{$1/2$} (max2);
\draw[->] (av4) to node[above, font=\small]{$1/2$} (max4);
\draw[->] (max4) to (min3);

\draw[->] (av23) to node[left, font=\small]{$1/2$} (av4);
\draw[->] (av23) to node[below, font=\small]{$1/2$} (av5);

\draw[->] (av5) to[in = 45, out = 135] node[above, font=\small]{$1/2$} (av23);

\end{tikzpicture}
\end{minipage}\hfill
\begin{minipage}{0.45\textwidth}
\centering
\begin{tikzpicture}[scale=0.8,>=stealth',max/.style={draw,rectangle,minimum size=0.5cm},min/.style={draw,circle,minimum size=0.5cm},av/.style={draw,diamond,minimum size=0.5cm}] 

\node[max] (max1) at (-2, 5) {$1$};

\node[min] (min3) at (5, 0) {$3$};
\node[min] (min2) at (2, -2) {$2$};

\node[min] (min4) at (1, 5) {$4$};

\node[max] (max2) at (-1, -2) {$8$};
\node[max] (max3) at (4, 1.5) {$9$};
\node[max] (max4) at (1, 0) {$4$};

\node[max] (max5) at (-1, 3.5) {$5$};
\node[min] (min5) at (-1, 1.5) {$5$};

\node[max] (max6) at (1, 2) {$6$};
\node[min] (min6) at (3, 2.5) {$6$};

\node[min] (min7) at (3, 4) {$7$};
\node[max] (max7) at (5, 4) {$7$};

\node[av] (av23) at (1, 3.75){};
\node[av] (av4) at (-1, 0){};
\node[av] (av5) at (5, 2.5){};

\draw[->] (max1) to node[above=-0.7ex, font=\small]{$4/3$} (min4);
\draw[->] (min4) to (av23);

\draw[->] (av23) to node[above, font=\small]{$1/2$} (max5);
\draw[->] (av23) to node[left, font=\small]{$1/2$} (max6);
\draw[->] (av5) to node[below right, font=\small]{$1/2$} (max3);

\draw[->] (av4) to node[left, font=\small]{$1/2$} (max2);
\draw[->] (av4) to node[above, font=\small]{$1/2$} (max4);

\draw[->] (max5) to (min5);
\draw[->] (min5) to (av4);
\draw[->] (max6) to (min6);
\draw[->] (min6) to (av5);
\draw[->] (max7) to (min7);
\draw[->] (min7) to (av23);

\draw[->] (max2) to (min2);
\draw[->] (max3) to (min3);
\draw[->] (max4) to (min3);

\draw[->] (av5) to node[right, font=\small]{$1/2$} (max7);

\end{tikzpicture}
\end{minipage}

\caption{The transformation of \cref{lemma:zwick_paterson,lemma:first_transformation,lemma:second_transformation} applied to one Random vertex from the graph presented in \cref{fig:graph}. Top left: the initial graph. Top right: the graph after the application of \cref{lemma:zwick_paterson}. Bottom left: the graph after the application of \cref{lemma:first_transformation}. Bottom right: the graph after the application of \cref{lemma:second_transformation}.}\label{fig:transformations}

\end{figure}

\begin{example}
Take the graph from \cref{fig:graph} and consider the Random vertex that has Min vertices $2$ and $3$ as its neighbors. \Cref{fig:transformations} presents the outcome of the procedure described in the lemmas above when applied to this vertex.
\end{example}

\section{General case of the tropical Helton--Nie conjecture}\label{section:general}

We now generalize \cref{lemma:main_stratum_cone} to tropically convex sets in $\trop^{n}$. In order to study this case, we use the notion of homogenization of a convex set. There are many possible homogenizations of a given set. We need to use three different notions.

\begin{definition}
If $\slin$ is a tropically convex set with only finite points (\ie, $\slin \subset \R^{n}$), then we define its \emph{real homogenization} as 
\[
\rhomslin = \{ (x_{0}, x_{0} + x) \in \ \R^{n+1} \colon x \in \slin \} \, .
\]
The set $\rhomslin$ is a real tropical cone.
If $\slin \subset \trop^{n}$ is a tropically convex set, then we define its \emph{homogenization} as 
\[
\homslin = \{ (x_{0}, x_{0} + x) \in \ \trop^{n+1} \colon x \in \slin \} \, .
\]
The set $\homslin$ is a tropical cone. If $\spectra(Q^{(0)}|Q^{(1)},\dots,Q^{(n)}) \subset \trop^{n}$ is a tropical Metzler spectrahedron, then we define its \emph{formal homogenization} as the tropical Metzler spectrahedron $\fhomspectra \subset \trop^{n+1}$, $\fhomspectra \coloneqq \spectra(-\infty| Q^{(0)},Q^{(1)},\dots,Q^{(n)})$. The set $\fhomspectra$ is a tropical Metzler spectrahedral cone.
\end{definition}

\begin{lemma}\label{lemma:helton_nie_main_stratum}
Every closed semilinear tropically convex set in $\R^{n}$ is a projection of a tropical Metzler spectrahedron.
\end{lemma}
\begin{proof}
Take any closed semilinear tropically convex set $\slin \subset \R^{n}$ and consider its real homogenization $\rhomslin$. This is a closed semilinear real tropical cone in $\R^{n+1}$. By \cref{lemma:main_stratum_cone}, $\rhomslin$ is a projection of a tropical Metzler spectrahedron $\spectra_{1} \subset \trop \times \trop^{n} \times \trop^{n'}$. Consider the set
\[
\spectra_{2} = \{ (x_{0}, x, y) \in \trop \times \trop^{n} \times \trop^{n'} \colon (x_{0},x,y) \in \spectra_{1} \land x_{0} = 0 \} \, .
\]
The set $\spectra_{2}$ is a tropical Metzler spectrahedron. Furthermore, $\slin$ is its projection.
\end{proof}

We now want to extend this result to tropically convex sets in $\trop^{n}$. In order to do this, we proceed stratum-by-stratum. This requires us to show that a tropical convex hull of finitely many projected Metzler spectrahedra is a projected Metzler spectrahedron. In the classical case of real spectrahedra, it is known that a convex hull of finitely many projected spectrahedra is a projected spectrahedron. This fact has a very short proof presented in \cite{netzer_sinn_projected_spectrahedra}. The proof in the tropical case is exactly the same (we only change the classical notation to the tropical one). Let us present this proof for the sake of completeness. 

\begin{lemma}\label{lemma:projection_of_homogenization}
A tropically convex set $\slin \subset \trop^{n}$ is a projected tropical Metzler spectrahedron if and only if its homogenization is a projected tropical Metzler spectrahedron.
\end{lemma}
\begin{proof}
First, suppose that $\homslin$ is a projection of a tropical Metzler spectrahedron $\spectra_{1} \subset \trop \times \trop^{n} \times \trop^{n'}$. Consider the set
\[
\spectra_{2} = \{ (x_{0}, x, y) \in \trop \times \trop^{n} \times \trop^{n'} \colon (x_{0},x,y) \in \spectra_{1} \land x_{0} = 0 \} \, .
\]
The set $\spectra_{2}$ is a tropical Metzler spectrahedron. Furthermore, $\slin$ is its projection. Conversely, suppose that $\slin$ is a projection of a tropical Metzler spectrahedron $\spectra_{1} \subset \trop^{n} \times \trop^{n'}$. Consider its formal homogenization $\fhomspectra_{1} \subset \trop^{1 + n + n'}$ and take the set
\[
\spectra_{2}  = \{ (x_{0}, x, y, z) \in \trop \times \trop^{n} \times \trop^{n'} \times \trop^{n} \colon (x_{0},x, y) \in \fhomspectra_{1} \, \land \, \forall k \in [n], x_{0} + z_{k} \ge 2x_{k} \} \, .
\]
The set $\spectra_{2}$ is a tropical Metzler spectrahedron. We will show that $\homslin$ is a projection of $\spectra_{2}$. Take any point $x \in \slin$. Then, there exists $y$ such that $(0,x,y) \in \fhomspectra_{1}$. Therefore, $(x_{0}, x + x_{0}e, y + x_{0}e) \in \fhomspectra_{1}$ for any $x_{0} \in \R$. If we take $z_{k}$ large enough, then $(x_{0}, x + x_{0}e, y + x_{0}e,z) \in \spectra_{2}$. Moreover, we have $-\infty \in \spectra_{2}$. This shows that $\homslin$ is included in the projection of $\spectra_{2}$. Conversely, suppose that $(x_{0}, x, y, z) \in \spectra_{2}$. If $x_{0} = -\infty$, then $x = -\infty$ and hence $(x_{0}, x) \in \homslin$. If $x_{0} \neq -\infty$, then we have $(0, x - x_{0}e, y - x_{0}e, z - x_{0}e) \in \spectra_{2}$. Hence $(0, x - x_{0}e, y - x_{0}e) \in \fhomspectra_{1}$, $(x - x_{0}e, y - x_{0}e) \in \spectra_{1}$, and $x - x_{0}e \in \slin$. Therefore $(x_{0}, x) \in \homslin$.
\end{proof}

\begin{lemma}\label{lemma:union_of_projections}
Suppose that $\slin_{1}, \slin_{2} \subset \trop^{n}$ are projected tropical Metzler spectrahedra. Then $\tconv(\slin_{1} \cup \slin_{2})$ is a projected tropical Metzler spectrahedron.
\end{lemma}
\begin{proof}
Let $\slin = \tconv(\slin_{1} \cup \slin_{2})$ and consider
\[
\homslin_{1} \tplus \homslin_{2} = \{ x \in \trop^{n+1} \colon \exists (u,w) \in \homslin_{1} \times \homslin_{2}, \, x = u \tplus w \} \, .
\]
Observe that we have the identity $\homslin = \homslin_{1} \tplus \homslin_{2}$. Indeed, since $\slin_{1} \subset \slin$, we have $\homslin_{1} \subset \homslin$. Similarly, $\homslin_{2} \subset \homslin$. Therefore, we have $\homslin_{1} \tplus \homslin_{2} \subset \homslin$. Conversely, take a point $z \in \homslin$. By \cref{convex_hull_of_union}, we can write $z$ as
\[
z = \Bigl( z_{0}, z_{0} \tdot \bigl( (\lambda \tdot x) \tplus (\mu \tdot y) \bigl) \Bigr) \in \homslin \, ,
\]
where $\lambda \tplus \mu = 0$, $x \in \slin_{1}$, and $y \in \slin_{2}$. Then $z = \tilde{x} \tplus \tilde{y}$, where 
\begin{equation*}
\begin{aligned}
\tilde{x} &= (\lambda \tdot z_{0}, (\lambda \tdot z_{0}) \tdot x) \in \homslin_{1} \, , \\
\tilde{y} &= (\mu \tdot z_{0}, (\mu \tdot z_{0}) \tdot y) \in \homslin_{2} \, .
\end{aligned}
\end{equation*}
Hence $\homslin = \homslin_{1} \tplus \homslin_{2}$ and the claim follows from \cref{lemma:projection_of_homogenization}.
\end{proof}

We are now ready to present the proof of \cref{theorem:tropical_convex_sets}.

\begin{proof}[Proof of \cref{theorem:tropical_convex_sets}]
The equivalence between \cref{main:item1} and \cref{main:item2} is given in \cref{tropical_conv_salg}. The implication from \cref{main:item2} to \cref{main:item3} follows from \cref{lemma:helton_nie_main_stratum}. We now prove the implication from \cref{main:item3} to \cref{main:item4}. Let $\slin \subset \trop^{n}$ be as in \cref{main:item3}. If $\slin$ is empty, then it is a tropical Metzler spectrahedron defined by a single inequality $-\infty \ge 0$. Otherwise let $K \subset [n]$ be any nonempty set such that the stratum $\slin_{K} \subset \R^{\card{K}}$ is nonempty. The set $\slin_{K}$ is a projection of a tropical Metzler spectrahedron $\spectra_{K} \subset \trop^{\card{K}} \times \trop^{n'}$. For any $x \in \trop^{n}$ we denote by $x_{K} \in \trop^{\card{K}}$ the subvector formed by the coordinates of $x$ with indices in $K$. Furthermore, let $X_{K}\subset \trop^{n}$ denote the set
\[
X_{K} = \{ x \in \trop^{n} \colon x_{k} \neq -\infty \iff k \in K\} \, .
\]
The set $\slin \cap X_{K}$ is a projection of a tropical Metzler spectrahedron defined as
\[
\tilde{\spectra}_{K} = \{ (x,y) \in \trop^{n} \times \trop^{n'} \colon (x_{K},y) \in \spectra_{K} \, \land \, \forall k \notin K, \, -\infty \ge x_{k} \} \, .
\]
Moreover, for $K = \emptyset$, let us denote $X_{\emptyset} = -\infty$. Note that the intersection $\slin \cap X_{\emptyset}$ is either empty or is equal to $-\infty$, and that $-\infty$ is a tropical Metzler spectrahedron (defined by the inequalities $-\infty \ge x_{k}$ for all $k \in [n]$). Hence, we have $\slin = \cup_{K \subset [n]} \, \slin \cap X_{K} = \tconv(\cup_{K \subset [n]} \slin \cap X_{K})$. Therefore, the claim follows from \cref{lemma:union_of_projections}. Finally, to prove the implication \cref{main:item4} to \cref{main:item1}, let $\slin \subset \trop^{n}$ be a projection of a tropical Metzler spectrahedron $\spectra \subset \trop^{n} \times \trop^{n'}$. By \cref{metzler_is_spectra}, there is a spectrahedron $\bspectra \in \nnpuiseux^{n + n'}$ such that $\val(\bspectra) = \spectra$. Let $\bo \pi \colon \puiseux^{n+n'} \to \puiseux^{n}$ denote the projection on the first $n$ coordinates. Then $\slin = \val(\bo \pi(\bspectra))$.
\end{proof}

\begin{remark}\label{uniformity}
Consider a convex semialgebraic subset $\bsalg$ over $\puiseux^n$. \Cref{th-main} shows that there exist integers $p$ and $m$ such that $\bsalg$ has the same image by the valuation as a projection of some spectrahedron over $\puiseux^{p}$ associated with matrices of size $m \times m$. The integers $p$ and $m$ appearing in the proof of this theorem have the following remarkable uniformity property: 
if $\bsalg$ is given as a union of finitely many basic semialgebraic sets of the form~\cref{eq:basic}, then $p, m$ are bounded from above by a number $N$ that depends only on the degrees and the number of polynomials involved in the description of $\bsalg$ (i.e., that $N$ is independent of the coefficients of these polynomials). The proof, however, is quite involved. First, one should observe that, given only the degrees and the number of polynomials describing $\bsalg$, the Denef--Pas quantifier elimination creates a finite set of $\logroups$-formulas such that every stratum of $\val(\bspectra)$ is described by a formula from this set. Second, \cref{theorem:tropical_convex_sets} gives a tropical Metzler spectrahedron $\spectra$ such that $\val(\bspectra)$ is its projection. A careful examination of the proof presented here shows that the dimension and the size of the matrices defining $\spectra$ can be bounded by a quantity that depends only on the aforementioned $\logroups$-formulas (and not on the particular choice of their parameters). This gives the desired bound $N$.
\end{remark}
\begin{remark}\label{rk-anyvaluedfield}
Given the bound of~\cref{uniformity}, the Denef--Pas quantifier elimination implies that our main result (\cref{th-main}) is valid not only over the field of Puiseux series considered here, but over every real closed valued field equipped with a nontrivial and convex valuation.
\end{remark}

\section{Concluding remarks}
We showed that the convex semialgebraic sets and the projections of spectrahedra over the nonarchimedean field of real Puiseux series have the same images by the nonarchimedean valuation. We gave an explicit representation for these images, as the subfixed point sets of semilinear monotone homogeneous maps 
(dynamic programming operators of zero-sum stochastic games with perfect information).
One may ask 
whether more insight on the projections
of spectrahedra over nonarchimedean fields or over the field of real numbers can be gotten by tropical methods. In this respect,
we note that we considered the simplest possible tropicalization,
looking at the image of Puiseux series by their ordinary valuations. 
We also leave it as a further work to see whether more sophisticated tropicalizations, capturing also the sign, or  higher order approximations of Puiseux series (spaces of jets) may
be exploited. 

\bibliographystyle{alpha}

\end{document}